\documentclass[leqno,final]{siamltex}
\usepackage{amsmath}
\usepackage{graphicx}
\usepackage{amssymb}

\newtheorem{remark}{Remark}

\newcommand{\hP}{\widehat{P}}
\newcommand{\FE}{\mbox{\tiny FE}}
\newcommand{\DG}{\mbox{\tiny DG}}
\newcommand{\cE}{\mathcal{E}}
\newcommand{\cT}{\mathcal{T}}

\newcommand{\mce}{\mathcal{E}_h}
\newcommand{\mct}{\mathcal{T}_h}
\newcommand{\bl}{\big\langle}
\newcommand{\br}{\big\rangle}

\newcommand{\eps}{\epsilon}

\newcommand{\Ome}{\Omega}
\newcommand{\p}{\partial}
\newcommand{\nab}{\nabla}

\def\esssupI{\underset{t\in [0,\infty)}{\mbox{\rm ess sup }}}
\begin{document}

\title{Analysis of interior penalty discontinuous Galerkin methods
for the Allen-Cahn equation and the mean curvature flow}
\markboth{XIAOBING FENG AND YUKUN LI}{DG METHODS FOR ALLEN-CAHN EQUATION}

\author{
Xiaobing Feng\thanks{Department of Mathematics, The University of Tennessee, 
Knoxville, TN 37996, U.S.A. ({\tt xfeng@math.utk.edu.})
The work of the this author was partially supported by
the NSF grants DMS-1016173 and DMS-1318486.}
\and
Yukun Li\thanks{Department of Mathematics, The University of Tennessee, 
Knoxville, TN 37996, U.S.A. ({\tt yli@math.utk.edu.}) 
The work of the this author was partially supported by
the NSF grants DMS-1016173 and DMS-1318486.} 
}

\maketitle

\begin{abstract}
This paper develops and analyzes two fully discrete  
interior penalty discontinuous Galerkin (IP-DG) methods
for the Allen-Cahn equation, which is a nonlinear singular  
perturbation of the heat equation and originally arises from
phase transition of binary alloys in materials science,
and its sharp interface limit (the mean curvature flow) as
the perturbation parameter tends to zero. Both fully implicit 
and energy-splitting time-stepping schemes are proposed. 
The primary goal of the paper is to derive sharp error 
bounds which depend on the reciprocal of the perturbation 
parameter $\epsilon$ (also called ``interaction length") 
only in some lower polynomial order, instead of exponential order,
for the proposed IP-DG methods. The derivation is based on
a refinement of the nonstandard error analysis technique first 
introduced in \cite{Feng_Prohl03}.  The 
centerpiece of this new technique is to establish a spectrum estimate 
result in totally discontinuous DG finite element spaces with a help of a similar
spectrum estimate result in the conforming finite element spaces
which was established in \cite{Feng_Prohl03}. 
As a nontrivial application of the sharp error estimates, 
they are used to establish convergence and the rates of convergence 
of the zero-level sets of the fully discrete IP-DG solutions 
to the classical and generalized mean curvature flow. 
Numerical experiment results are also presented to 
gauge the theoretical results and the performance of the 
proposed fully discrete IP-DG methods.
 \end{abstract}



\section{Introduction}\label{sec-1}
The singular perturbation of the heat equation to be considered in 
this paper has the form 
\begin{equation}\label{eq1.1}
u_t-\Delta u+\frac{1}{\epsilon^2}f(u)=0  \qquad \mbox{in } 
\Omega_T:=\Omega\times(0,T),
\end{equation}
where $\Omega\subseteq \mathbf{R}^d \,(d=2,3)$ is a bounded domain 
and $f=F'$ for some double well potential density function $F$.
In this paper we focus on the following widely used quartic  density function:
\begin{equation}\label{eq1.2}
F(u)=\frac{1}{4}(u^2-1)^2.
\end{equation}

Equation \eqref{eq1.1}, which is known as the Allen-Cahn equation in the 
literature, was originally introduced by Allen and Cahn in \cite{Allen_Cahn79}
as a model to describe the phase separation process of a binary alloy at a
fixed temperature. In the equation $u$ denotes the concentration of one of 
the two species of the alloy, and $\epsilon$ represents the
interaction length. We remark that equation \eqref{eq1.1} differs from the 
original Allen-Cahn equation in the scaling of the time, $t$ here 
represents $\frac{t}{\epsilon^2}$ in the original formulation, hence, 
it is a fast time. To completely describe the physical (and mathematical)
problem, equation \eqref{eq1.1} must be complemented with appropriate 
initial and boundary conditions. The following boundary and initial 
conditions will be considered in this paper:
\begin{alignat}{2}\label{eq1.3}
\frac{\partial u}{\partial n} &=0 &&\qquad \mbox{in }
\partial\Omega_T:=\partial\Omega\times(0,T), \\
u &=u_0 &&\qquad \mbox{in }\Omega\times\{t=0\}. \label{eq1.4}
\end{alignat}

In addition to the important role it plays in materials phase transition, the 
Allen-Cahn equation has also been well-known and intensively studied in the 
past thirty years due to its connection to the celebrated curvature driven 
geometric flow known as {\em the mean curvature flow~} or {\em the motion by 
mean curvature} (cf. \cite{ESS92, Ilmanen93} and the references therein). 
It was proved that \cite{ESS92} the zero-level set $\Gamma_t^\epsilon
:=\{ x\in \Omega; u(x, t)=0 \}$ of the solution $u$ to the 
problem \eqref{eq1.1}--\eqref{eq1.4} converges to 
{\em the mean curvature flow} which refers to the evolution
of a curve/surface governed by the geometric law $V_n=\kappa$,
where $V_n$ and $\kappa$ respectively stand for the (inward) normal 
velocity and the mean curvature of the curve/surface. In fact,
the Allen-Cahn equation (and the related Cahn-Hilliard equation) 
has emerged as a fundamental equation as well as a building block 
in the {\em phase field methodology} or the {\em diffuse interface methodology}
for moving interface and free boundary problems arising from 
various applications such as fluid dynamics, materials science,
image processing and biology (cf \cite{FHL07, McFadden02} 
and the references therein). The diffuse interface 
method provides a convenient mathematical formalism for  
numerically approximating the moving interface problems because 
there is no need to explicitly 
compute the interface in the diffuse interface formulation. The 
biggest advantage of the diffuse interface method is its ability
to handle with ease singularities of the interfaces. Computationally,
like many singular perturbation problems, the main issue is to  
resolve the (small) scale introduced by the parameter $\epsilon$ in the 
equation. The problem could become intractable, especially 
in three-dimensional case if uniform meshes are used. This difficulty is 
often overcome by exploiting the predictable (at least for small 
$\epsilon$) PDE solution profile and by using adaptive mesh techniques 
(cf. \cite{KNS04,Feng_Wu05}) so fine meshes
are only used in a small neighborhood of the phase front.

Numerical approximations of the Allen-Cahn equation have been 
extensively investigated in the past thirty years (cf. 
\cite{Bartels_Muller_Ortner09,Elliott97,Feng_Prohl03} 
and the references therein). However,
most of these works were carried out for a fixed parameter 
$\epsilon$. The error estimates, which are obtained using the 
standard Gronwall inequality technique, show an exponential dependence 
on $\frac{1}{\epsilon}$.  Such an estimate is clearly not useful
for small $\epsilon$, in particular, in addressing the issue 
whether the flow of the computed numerical interfaces converge to the 
original sharp interface model: the mean curvature flow. 
Better error estimates should only depend on $\frac{1}{\epsilon}$
in some (low) polynomial orders because they can be used to provide 
an answer to the above convergence issue. In fact, such an estimate is 
the best result (in terms of $\epsilon$) one can expect. 
The first such polynomial order in $\frac{1}{\epsilon}$ a priori 
estimate was obtained by Feng and Prohl in \cite{Feng_Prohl03} 
for standard finite element approximations of the Allen-Cahn
problem \eqref{eq1.1}--\eqref{eq1.4}. 
Extensions of the results of \cite{Feng_Prohl03}, in particular, the sensitivity 
of the eigenvalue to the topology was later considered, and some numerical tests 
were also given by Bartels et al. in \cite{Bartels_Muller_Ortner09}. 
In addition, polynomial order in $\frac{1}{\epsilon}$ a posteriori error 
estimates were obtained in \cite{KNS04,Feng_Wu05,Bartels_Muller_Ortner09}. 
One of the key ideas employed in all these works is 
to use a nonstandard error estimate technique which is based 
on establishing a discrete spectrum estimate (using its 
continuous counterpart) for the linearized Allen-Cahn operator.  
An immediate application of the polynomial order in $\frac{1}{\epsilon}$ 
a priori and a posteriori error estimates is to prove the convergence 
of the numerical interfaces of the underlying finite element approximations
to the mean curvature flow as $\epsilon$ and mesh
sizes $h$ and $\tau$ all tend to zero, and to establish rates of convergence  
(in powers of $\epsilon$) for the numerical interfaces before 
the onset of singularities of the mean curvature flow.

The primary objectives of this paper are twofold: First, we want to develop 
some interior penalty discontinuous Galerkin (IP-DG) methods and
to establish polynomial order in $\frac{1}{\epsilon}$ a priori error 
estimates as well as to prove convergence and rates of convergence 
for the IP-DG numerical interfaces. This goal is motivated by the 
advantages of DG methods in regard to designing adaptive mesh
methods and algorithms, which is an indispensable strategy 
with the diffuse interface methodology. Second, we use the 
Allen-Cahn equation as a prototype to develop new analysis 
techniques for analyzing convergence of numerical
interfaces to the sharp interface for DG (and nonconforming finite 
element) discretizations of phase field models. To the best of 
our knowledge, no such convergence result and analysis technique 
is available in the literature. The main obstacle for 
adapting the techniques of \cite{Feng_Prohl03} is that 
the DG (and nonconforming finite element) spaces are not 
subspaces of $H^1(\Omega)$. As a result, whether the desired 
discrete spectrum estimate holds becomes a key question to answer. 

The remainder of this paper is organized as follows. In section \ref{sec-2} 
we first recall some facts about the Allen-Cahn equation. In particular, 
we cite the spectrum estimate for the linearized Allen-Cahn operator from 
\cite{Chen94} and a nonlinear discrete Gronwall inequality from \cite{Pachpatte}. 
In section \ref{sec-3} we present two fully nonlinear IP-DG methods
for problem \eqref{eq1.1}--\eqref{eq1.4} with the implicit Euler time 
stepping for the linear terms. 
The two methods differ in how the nonlinear term 
is discretized in time.  The first is fully implicit and the second 
uses a well-known energy splitting idea due to Ere \cite{Eyre00}.  The rest 
of section \ref{sec-3} devotes to the convergence analysis of the proposed 
IP-DG methods.  The highlights of analysis include establishing a
discrete spectrum estimate for the linearized Allen-Cahn operator in DG 
spaces and deriving optimal order (in $h$ and $\tau$) and polynomial 
order in $\frac{1}{\epsilon}$ a priori error estimates for the proposed 
IP-DG methods. In section \ref{sec-4}, using the error estimates 
of section \ref{sec-3} we prove the convergence and rates of 
convergence for the numerical interfaces of the IP-DG solutions to the 
sharp interface of the mean curvature flow. Finally, we present 
some numerical experiment results in section \ref{sec-5} to gauge 
the performance of the proposed fully discrete IP-DG methods.

\section{Preliminaries}\label{sec-2}
In this section, we first recall a few facts about the solution of 
the problem \eqref{eq1.1}--\eqref{eq1.4} which can be found in 
\cite{Feng_Prohl03, Chen94}. These facts will be used in the 
analysis of section \ref{sec-3} and \ref{sec-4}. We then cite a 
lemma which provides an upper bound for discrete sequences that satisfy a 
Bernoulli-type inequality, and this lemma is crucially used in our 
error analysis in section \ref{sec-3}.
Standard function and space notations are adopted in this 
paper. $(\cdot,\cdot)_\Ome$ denotes the standard inner product on $L^2(\Ome)$,
$C$ and $c$ denote generic positive constants which is independent 
of $\epsilon$, space and time step sizes  $h$ and $\tau$.
We begin by recalling a well-known fact \cite{ESS92,Ilmanen93} that the 
Allen-Cahn equation \eqref{eq1.1} can be interpreted as the $L^2$-gradient 
flow for the following Cahn-Hilliard energy functional
\begin{equation}\label{eq2.1}
J_\epsilon(v):= \int_\Omega \Bigl( \frac12 |\nabla v|^2
+ \frac{1}{\epsilon^2} F(v) \Bigr)\, dx
\end{equation}

In order to derive a priori solution estimates, as in \cite{Feng_Prohl03} 
we make the following assumptions on the initial datum $u_0$.

\medskip
{\bf General Assumption} (GA) 
\begin{itemize}
\item[(1)] There exists a nonnegative constant $\sigma_1$ such that
\begin{equation}\label{eq2.2}
J_{\epsilon}(u_0)\leq C\epsilon^{-2\sigma_1}.
\end{equation}
\item[(2)] There exists a nonnegative constant $\sigma_2$ such that
\begin{equation}\label{eq2.3}
\|\Delta u_0 -\epsilon^{-2} f(u_0)\|_{L^2(\Omega)} \leq C\epsilon^{-\sigma_2}.
\end{equation}

\item[(3)]
There exists nonnegative constant $\sigma_3$ such that
\begin{equation}\label{eq2.4}
\lim_{s\rightarrow0^{+}} \|\nabla u_t(s)\|_{L^2(\Omega)}\leq C\epsilon^{-\sigma_3}.
\end{equation}

\end{itemize}

The following solution estimates can be found in \cite{Feng_Prohl03}.

\begin{proposition}\label{prop2.1}
Suppose that \eqref{eq2.2} and \eqref{eq2.3} hold. Then the solution $u$ of 
problem \eqref{eq1.1}--\eqref{eq1.4} satisfies the following estimates:
\begin{align} \label{eq2.5}
&\esssupI \|u(t)\|_{L^\infty(\Ome)} \leq \|u_0\|_{L^\infty(\Ome)},\\
&\esssupI\, J_{\epsilon}(u) 
+ \int_{0}^{\infty} \|u_t(s)\|_{L^2(\Omega)}^2\, ds 
\leq C \eps^{-2\sigma_1},\label{eq2.5b} \\
&\int_{0}^{T} \|\Delta u(s)\|^2\, ds \leq C \eps^{-2(\sigma_1+1)}, \label{eq2.6}\\
&\esssupI \Bigl( \|u_t\|_{L^2(\Omega)}^2 +\|u\|_{H^2(\Omega)}^2 \Bigr)
+\int_{0}^{\infty} \|\nabla u_t(s)\|_{L^2(\Omega)}^2\, ds
\leq C \eps^{-2\max\{\sigma_1+1,\sigma_2\}}, \label{eq2.7} \\
&\int_{0}^{\infty} \Bigl(\|u_{tt}(s)\|_{H^{-1}(\Omega)}^2
+\|\Delta u_t(s)\|_{H^{-1}(\Omega)}^2\Bigr) \, ds
\leq C \eps^{-2\max\{\sigma_1+1,\sigma_2\}}. \label{eq2.8}
\end{align}
In addition to \eqref{eq2.2} and \eqref{eq2.3}, 
suppose that \eqref{eq2.4} holds, then $u$ also satisfies
\begin{align} \label{eq2.9}
&\esssupI \|\nabla u_t\|_{L^2(\Ome)}^2 +\int_0^{\infty} \|u_{tt}(s)\|_{L^2}^2 \,ds
\leq C\eps^{-2\max\{\sigma_1+2,\sigma_3\}},\\
&\int_{0}^{\infty} \|\Delta u_t(s)\|_{L^2(\Ome)}^2 \,ds
\leq C\eps^{-2\max\{\sigma_1+2, \sigma_3\}}. \label{eq2.10}
\end{align}

\end{proposition}

Next, we quote a lower bound estimate for the principal eigenvalue of
the following linearized Allen-Cahn operator:
\begin{equation}\label{eq2.11}
\mathcal{L}_{\mbox{\tiny AC}}:=-\Delta + f'(u)I,
\end{equation}
where $I$ stands for the identity operator. 

\begin{proposition}\label{prop2.3}
Suppose that \eqref{eq2.2} and \eqref{eq2.3} hold, and $u_0$ satisfies
some profile as described in \cite{Chen94}. Let $u$ denote the 
solution of problem \eqref{eq1.1}--\eqref{eq1.4}.  Then there exists a 
positive $\eps$-independent constant $C_0$ 
such that the principle eigenvalue of the linearized Allen-Cahn 
operator ${\mathcal L}_{\mbox{\tiny AC}}$ satisfies for $0<\eps<<1$
\begin{equation}\label{eq2.12}
\lambda_{\mbox{\tiny AC}} \equiv \inf_{\psi \in H^1(\Ome)\atop \psi \neq 0}
\frac{ \|\nab \psi\|_{L^2(\Ome)}^2 + \eps^{-2}\, \bigl(f'(u) \psi, \psi\bigr)}
{\|\psi\|_{L^2(\Ome)}^2} \geq -C_0.
\end{equation}

\end{proposition}

\begin{remark}
(a) A proof of Proposition \ref{prop2.3} can be found in \cite{Chen94}. 
A discrete generalization of \eqref{eq2.12} on $C^0$ finite element spaces was 
proved in \cite{Feng_Prohl03}. It plays a pivotal role in the nonstandard
convergence analysis of \cite{Feng_Prohl03}. In the next section, we shall 
prove another discrete generalization of \eqref{eq2.12} on
DG finite element spaces.

(b) The restriction on the initial function $u_0$ is needed to guarantee 
that the solution $u(t)$ satisfies certain profile at later time $t>0$ which is required 
in the proof of \cite{Chen94}. One example of admissible initial functions is 
$u_0=tanh(\frac{d_0(x)}{\eps})$, where $d_0(x)$ stands for the signed distance 
function to the initial interface $\Gamma_0$.
\end{remark}

\medskip
The classical Gronwall lemma derives an estimate for any function 
which satisfies a first order linear differential inequality.
It is a main technique for deriving error estimates for continuous-in-time
semi-discrete discretizations of many initial-boundary value PDE problems.
Similarly, the discrete counterpart of Gronwall lemma is a main
technical tool for deriving error estimates for fully discrete schemes.
However, for many nonlinear PDE problems, the classical Gronwall lemma
does not apply because of nonlinearity, instead, some nonlinear generalization
must be used. In case of the power (or Bernoulli-type) nonlinearity, 
a generalized Gronwall lemma was proved in \cite{Feng_Wu05}. In the 
following we state a discrete counterpart of the lemma in \cite{Feng_Wu05}, 
and the proof of a similar lemma can be found in \cite{Pachpatte}. This 
lemma will be utilized crucially in the next section.   

\begin{lemma}\label{lem2.1}
Let $\{S_{\ell} \}_{\ell\geq 1}$ be a positive nondecreasing sequence and 
$\{b_{\ell}\}_{\ell\geq 1}$ and $\{k_{\ell}\}_{\ell\geq 1}$ be nonnegative sequences, 
and $p>1$ be a constant. If
\begin{eqnarray}\label{eq2.13}
&S_{\ell+1}-S_{\ell}\leq b_{\ell}S_{\ell}+k_{\ell}S^p_{\ell} \qquad\mbox{for } \ell\geq 1,
\\ \label{eq2.14}
&S^{1-p}_{1}+(1-p)\mathop{\sum}\limits_{s=1}^{\ell-1}k_{s}a^{1-p}_{s+1}>0
\qquad\mbox{for } \ell\geq 2,
\end{eqnarray}
then
\begin{equation}\label{eq2.15}
S_{\ell}\leq \frac{1}{a_{\ell}} \Bigg\{S^{1-p}_{1}+(1-p)
\sum_{s=1}^{\ell-1}k_{s}a^{1-p}_{s+1}\Bigg\}^{\frac{1}{1-p}}\qquad\text{for}\ \ell\geq 2,  
\end{equation}
where
\begin{equation}\label{eq2.16}
a_{\ell} := \prod_{s=1}^{\ell-1} \frac{1}{1+b_{s}} \qquad\mbox{for } \ell\geq 2.
\end{equation}
\end{lemma}

\section{Fully discrete IP-DG approximations}\label{sec-3}

\subsection{Formulations} \label{sec-3.1}
Let $\cT_h$ be a quasi-uniform ``triangulation" of $\Omega$ such that 
$\overline{\Ome}=\bigcup_{K\in\cT_h} \overline{K}$. Let $h_K$ denote
the diameter of $K\in \cT_h$ and $h:=\mbox{max}\{h_K; K\in\cT_h\}$. 
We recall that the standard broken Sobolev space $H^s(\cT_h)$ and DG finite 
element space $V_h$ are defined as 
\[
H^s(\mathcal{T}_h):=\prod_{K\in\cT_h} H^{s}(K), \qquad
V_h:=\prod_{K\in\cT_h} P_r(K),
\]
where $P_r(K)$ denotes the set of all polynomials whose degrees do not 
exceed a given positive integer $r$.  Let $\mce^I$ denote the set of all
interior faces/edges of $\mct$, $\mce^B$ denote the set of all boundary
faces/edges of $\mct$, and $\mce:=\mce^I\cup \mce^B$. The $L^2$-inner product 
for piecewise functions over the mesh $\cT_h$ is naturally defined by
\[
(v,w)_{\mathcal{T}_h}:= \sum_{K\in \cT_h} \int_{K} v w\, dx,
\]
and for any set $\mathcal{S}_h \subset \mce$, the $L^2$-inner product 
over $\mathcal{S}_h$ is defined by
\begin{align*}
\bl v,w\br_{\mathcal{S}_h} :=\sum_{e\in \mathcal{S}_h} \int_e vw\, ds.
\end{align*}

Let $K, K'\in \cT_h$ and $e=\partial K\cap \partial K'$ and assume 
global labeling number of $K$ is smaller than that of $K'$.  
We choose $n_e:=n_K|_e=-n_{K'}|_e$ as the unit normal on $e$ and 
define the following standard jump and average notations
across the face/edge $e$:
\begin{alignat*}{4}
[v] &:= v|_K-v|_{K'} 
\quad &&\mbox{on } e\in \mce^I,\qquad
&&[v] :=v\quad 
&&\mbox{on } e\in \mce^B,\\
\{v\} &:=\frac12\bigl( v|_K +v|_{K'} \bigr) \quad
&&\mbox{on } e\in \mce^I,\qquad
&&\{v\}:=v\quad 
&&\mbox{on } e\in \mce^B
\end{alignat*}
for $v\in V_h$. 

Let $M$ be a (large) positive integer. Define $\tau:=T/M$ and $t_m:=m\tau$
for $m=0,1,2,\cdots,M$ be a uniform partition of $[0,T]$. For a sequence
of functions $\{v^m\}_{m=0}^M$, we define the (backward) difference operator   
\[
d_t v^m:= \frac{v^m-v^{m-1}}{k}, \qquad m=1,2,\cdots,M.
\]

We are now ready to introduce our fully discrete DG finite element methods
for problem \eqref{eq1.1}--\eqref{eq1.4}. They are defined by seeking
$u_h^m\in V_h$ for $m=0,1,2,\cdots, M$ such that
\begin{alignat}{2}\label{eq3.1}
\bigl( d_t u_h^{m+1},v_h\bigr)_{\mct}
+a_h(u_h^{m+1},v_h)+\frac{1}{\eps^2}\bigl( f^{m+1},v_h\bigr)_{\mct} &=0
&&\quad\forall v_h\in V_h,
\end{alignat}
where
\begin{align}\label{eq3.3}
a_h(w_h,v_h) &:=\bigl( \nabla w_h,\nabla v_h\bigr)_{\mct}
-\bigl\langle \{\p_n w_h\}, [v_h] \bigr\rangle_{\mce^I}  \\
&\hskip 1.1in
+\lambda \bigl\langle [w_h], \{\p_n v_h\} \bigr\rangle_{\mce^I} + j_h(w_h,v_h), \nonumber\\
j_h(w_h,v_h)&:=\sum_{e\in\mce^I}\frac{\sigma_e}{h_e}\bl [w_h],[v_h] \br_e,  
\label{eq3.4} \\
f^{m+1}&:= (u_h^{m+1})^3-u_h^m \quad\mbox{or}\quad 
f^{m+1}:= (u_h^{m+1})^3-u_h^{m+1},
\label{eq3.5}
\end{align}
where $\lambda=0,\pm1$ and $\sigma_e$ is a positive piecewise constant 
function on $\mce^I$, which will be chosen later (see Lemma \ref{lem-3.2}).  In addition,
we need to supply $u_h^0$ to start the time-stepping,
whose choice will be clear (and will be specified) later 
when we derive the error estimates in section \ref{sec-3.4}.

We conclude this subsection with a few remarks to explain the above 
IP-DG methods. 

\begin{remark}\label{rem-3.1}
(a) The mesh-dependent bilinear form $a_h(\cdot,\cdot)$ is a 
well-known IP-DG discretization of the negative Laplace operator 
$-\Delta$, see \cite{Riviere08}. 

(b) Different choices of $\lambda$ give different schemes. In this
paper we only focus on the symmetric case with $\lambda=-1$. Also, 
$\sigma_e$ is called the penalty constant. 

(c) The time discretization is the simple backward Euler method
for the linear terms. However, we shall prove in section \ref{sec-3.2} that
the treatment of the nonlinear term results in two implicit
schemes which have different stability properties with respect to 
$\eps$. We also note that only fully implicit scheme 
(i.e., $f^{m+1}=(u_h^{m+1})^3-u_h^{m+1}$) was considered in 
\cite{Feng_Prohl03}, and the resulted finite element method
was proved only conditionally stable there.
\end{remark}

\subsection{Discrete energy laws and well-posedness} \label{sec-3.2}
As a gradient flow, problem \eqref{eq1.1}--\eqref{eq1.4} enjoys an 
energy law which leads to the estimate \eqref{eq2.5b} and then the subsequent
estimates given in Proposition \ref{prop2.1}. One simple criterion 
for building a numerical method for problem \eqref{eq1.1}--\eqref{eq1.4} is
whether the method satisfies a discrete energy law which mimics the 
continuous energy law \cite{Feng_Prohl03,FHL07}. The goal of this 
subsection is to show that the IP-DG methods proposed in the previous 
subsection are either unconditionally energy stable when 
$f^{m+1}=(u_h^{m+1})^3-u_h^m$ or conditionally energy stable 
when $f^{m+1}=(u_h^{m+1})^3-u_h^{m+1}$.

First, we introduce three mesh-dependent energy functionals
which can be regarded as DG counterparts of the continuous 
Cahn-Hilliard energy $J_\eps$ defined in \eqref{eq2.1}.
\begin{align}\label{eq3.7a}
\Phi^h(v) &:=\frac{1}{2} \|\nab v\|_{L^2(\cT_h)}^2
-\bigl\langle \{\p_n v\}, [v] \bigr\rangle_{\cE_h^I} + \frac12 j_h(v,v) \qquad
\forall v\in H^2(\cT_h), \\ 
J_\eps^h(v) &:= \Phi^h(v) +\frac{1}{\eps^2} \bigl( F(v), 1\bigr)_{\cT_h}
\qquad \forall v\in H^2(\cT_h), \label{eq3.7b} \\
I_\eps^h(v) &:= \Phi^h(v) +\frac{1}{\eps^2} \bigl( F_c^+(v), 1\bigr)_{\cT_h}
\qquad \forall v\in H^2(\cT_h), \label{eq3.7c}
\end{align}
where $F(v)=\frac14 (v^2-1)^2$ and $F_c^+(v):= \frac14 (v^4+1)$. 

If we define $F_c^-(v):=\frac12 v^2$, then there holds
the convex decomposition $F(v)=F_c^+(v)-F_c^-(v)$. 
It is easy to check that $\Phi^h$ and $I_\eps^h$ are convex functionals 
but $J_\eps^h$ is not because $F$ is not convex. Moreover, we have

\begin{lemma}\label{lem3.1}
Let $\lambda =-1$ in \eqref{eq3.3}, then there holds for all $v_h,w_h\in V_h$
\begin{align}\label{eq3.8}
\Bigl(\frac{\delta \Phi^h(v_h)}{ \delta v_h}, w_h \Bigr)_{\cT_h}
&:=\lim_{s\to 0} \frac{\Phi^h(v_h+ s w_h)-\Phi^h(v_h) }{s}   
=a_h(v_h, w_h), \\
\Bigl(\frac{\delta J_\eps^h(v_h)}{ \delta v_h}, w_h \Bigr)_{\cT_h}
:&=\lim_{s\to 0} \frac{J_\eps^h(v_h+ s w_h)-J_\eps^h(v_h) }{s} \label{eq3.9} \\
&=a_h(v_h, w_h) +\frac{1}{\eps^2} \bigl(F^{\prime}(v_h), w_h \bigr)_{\cT_h},
\nonumber\\ 
\Bigl(\frac{\delta I_\eps^h(v_h)}{ \delta v_h}, w_h \Bigr)_{\cT_h}
:&=\lim_{s\to 0} \frac{I_\eps^h(v_h+ s w_h)-I_\eps^h(v_h) }{s} \label{eq3.10} \\ 
&=a_h(v_h, w_h) +\frac{1}{\eps^2} \bigl((F_c^+)^{\prime}(v_h), w_h \bigr)_{\cT_h}.
\nonumber
\end{align}
\end{lemma}

Since the proof is straightforward, we omit it to save space. 

\begin{remark}
We remark that \eqref{eq3.8}--\eqref{eq3.10} provide respectively
the representations of the Fr\'echet derivatives of the energy functionals 
$\Phi^h, J_\eps^h$ and $I_\eps^h$ in $V^h$. This simple 
observation is very helpful, it allows us to recast our DG formulations
in  \eqref{eq3.1}--\eqref{eq3.5} as a minimization/variation problem
at each time step. It is also a deeper reason why the proposed DG methods
satisfy some discrete energy laws to be proved below.
\end{remark}

\begin{lemma}\label{lem-3.2}
There exist constants $\sigma_0, \alpha>0$ such that for 
$\sigma_e>\sigma_0$ for all $e\in \cE_h$ there holds 
\begin{equation}\label{eq3.11a}
\Phi^h(v_h)\geq \alpha \|v_h\|_{1,\DG}^2 \qquad\forall v_h\in V_h,
\end{equation}
where
\begin{equation}\label{eq3.11b}
\|v_h\|_{1,\DG}^2 := \|\nab v_h\|_{L^2(\cT_h)}^2 + j_h(v_h,v_h).
\end{equation}
\end{lemma}

\begin{proof}
Inequality \eqref{eq3.11a} follows immediately from the following observation
\begin{equation}\label{eq3.12}
2\Phi^h(v_h)= a_h(v_h,v_h) \qquad\forall v_h\in V_h,
\end{equation}
and the well-known coercivity property of the DG bilinear form
$a_h(\cdot,\cdot)$ (cf. \cite{Riviere08}).
\end{proof}

We now are ready to state our discrete energy/stability estimates.

\begin{theorem}\label{thm-3.1}
Let $\{u_h^m\}$ be a solution of scheme \eqref{eq3.1}--\eqref{eq3.5}.
Then there exists $\sigma_0^{\prime}>0$ such that for $\sigma_e> \sigma_0^{\prime},
\forall e\in \cE_h$
\begin{equation}\label{eq3.12b}
J^h_\eps(u_h^\ell) + k\sum_{m=0}^\ell R^m_{\eps,h} \leq J^h_\eps(u_h^0) 
\qquad\mbox{for } 0\leq \ell \leq M,
\end{equation}
where
\begin{align}\label{eq3.12c}
R^m_{\eps,h} &:=\Bigl(1\pm \frac{k}{2\eps^2} \Bigr)
\|d_t u_h^{m+1}\|_{L^2(\cT_h)}^2 
+\frac{k}{4} \|\nab d_t u_h^{m+1}\|_{L^2(\cT_h)}^2\\
&\qquad
+\frac{k}{4} j_h\bigl(d_t u_h^{m+1},d_t u_h^{m+1}\bigr)
+\frac{k}{4\eps^2}\|d_t(|u_h^{m+1}|^2-1)\|_{L^2(\cT_h)}^2, \nonumber 
\end{align}
and the ``+" sign in the first term is taken when $f^{m+1}
=(u_h^{m+1})^3-u_h^m$ and ``-" sign is taken when $f^{m+1}
=(u_h^{m+1})^3-u_h^{m+1}$.
\end{theorem}

\begin{proof}
Setting $v=d_t u_{h}^{m+1}$ in \eqref{eq3.1} 
we get
\begin{align}\label{eq3.13}
\|d_t u_h^{m+1}\|_{L^2(\cT_h)}^2
+ a_h(u_h^{m+1}, d_t u_h^{m+1}) 
+\frac{1}{\eps^2} \bigl( f^{m+1},d_t u_h^{m+1} \bigr)_{\cT_h}=0.
\end{align}

By the algebraic identity 
$a (a-b)=\frac12(a^2-b^2) + \frac12 (a-b)^2$ we have
\begin{align}\label{eq3.14}
&a_h(u_h^{m+1}, \nab d_t u_h^{m+1}) 
=\frac12 d_t a_h(u_h^{m+1}, u_h^{m+1}) 
+\frac{k}{2} \Bigl( \|\nab d_t u_h^{m+1}\|_{L^2(\cT_h)}^2 \\
&\hskip 1.3in
+ 2 \bigl\langle \{d_t \p_n u_h^{m+1}\}, [d_t u_h^{m+1}] \bigr\rangle_{\cE_h^I}  
+ j_h\bigl(d_t u_h^{m+1},d_t u_h^{m+1}\bigr) \Bigl).\nonumber
\end{align}
It follows from the trace and Schwarz inequalities that
\begin{align}\label{eq3.15}
&2 \bigl\langle \{d_t \p_n u_h^{m+1}\}, [d_t u_h^{m+1}] \bigr\rangle_{\cE_h^I}  
\geq -2\|\{d_t \p_n u_h^{m+1}\}\|_{L^2(\cE_h^I)} \|[d_t u_h^{m+1}]\|_{L^2(\cE_h^I)}\\
&\hskip 1.6in
\geq -C h^{-\frac12} \|d_t\nab u_h^{m+1}\|_{L^2(\cT_h)}  \|[d_t u_h^{m+1}]\|_{L^2(\cE_h^I)} \nonumber\\
&\hskip 1.6in
\geq - \frac12 \|d_t\nab u_h^{m+1}\|_{L^2(\cT_h)}^2
-Ch^{-1} \|[d_t u_h^{m+1}]\|_{L^2(\cE_h^I)}^2. \nonumber
\end{align}
Then there exists $\sigma_1>0$ such that for $\sigma_e>\sigma_1$
\begin{align}\label{eq3.16}
a_h(u_h^{m+1}, \nab d_t u_h^{m+1})
&\geq \frac12 d_t a_h(u_h^{m+1}, u_h^{m+1}) \\
&\qquad 
+\frac{k}{4} \Bigl( \|\nab d_t u_h^{m+1}\|_{L^2(\cT_h)}^2
+ j_h\bigl(d_t u_h^{m+1},d_t u_h^{m+1}\bigr) \Bigl). \nonumber
\end{align} 

We now bound the third term on the left-hand side of \eqref{eq3.13} from below.  
We first consider the case $f^{m+1}=(u_h^{m+1})^3 - u_h^m$. To the end,
we write 
\begin{align*}
f^{m+1} 
&=u_h^{m+1}\bigl(|u_h^{m+1}|^2-1\bigr) +kd_tu_h^{m+1} \\
&=\frac12 \bigl( (u_h^{m+1}+u_h^m)+kd_t u_h^{m+1} \bigr)
\bigl(|u_h^{m+1}|^2-1\bigr)+kd_tu_h^{m+1}.
\end{align*}
A direct calculation then yields
\begin{align}\label{eq3.17}
\frac{1}{\eps^2}\bigl(f^{m+1},d_t u_h^{m+1}\bigr)_{\cT_h}
&\geq \frac{1}{4\eps^2} d_t\| |u_h^{m+1}|^2-1\|_{L^2(\cT_h)}^2 \\
&\quad
+\frac{k}{4\eps^2}\|d_t(|u_h^{m+1}|^2-1)\|_{L^2(\cT_h)}^2
+\frac{k}{2\eps^2}\|d_t u_h^{m+1}\|_{L^2(\cT_h)}^2. \nonumber
\end{align}

On the other hand, when $f^{m+1}= f(u_h^{m+1})=(u_h^{m+1})^3-u_h^{m+1}$, 
we have (cf. \cite{Feng_Prohl03})
\begin{align}\label{eq3.18}
\frac{1}{\eps^2}\bigl(f^{m+1},d_t u_h^{m+1}\bigr)_{\cT_h}
&\geq \frac{1}{4\eps^2} d_t\| |u_h^{m+1}|^2-1\|_{L^2(\cT_h)}^2 \\
&\quad
+\frac{k}{4\eps^2}\|d_t(|u_h^{m+1}|^2-1)\|_{L^2(\cT_h)}^2
-\frac{k}{2\eps^2}\|d_t u_h^{m+1}\|_{L^2(\cT_h)}^2. \nonumber
\end{align}

It follows from \eqref{eq3.13}, \eqref{eq3.16}, \eqref{eq3.12} 
and \eqref{eq3.17} (resp. \eqref{eq3.18}) that
\begin{align*}
&\Bigl(1\pm \frac{k}{2\eps^2} \Bigr) \|d_t u_h^{m+1}\|_{L^2(\cT_h)}^2
+d_t\Bigl( \Phi^h(u_h^{m+1})
+\frac{1}{4\eps^2}\| |u_h^{m+1}|^2-1\|_{L^2(\cT_h)}^2 \Bigr) \\
&+\frac{k}{4} \Bigl( \|\nab d_t u_h^{m+1}\|_{L^2(\cT_h)}^2
+j_h\bigl(d_t u_h^{m+1},d_t u_h^{m+1}\bigr)
+\frac{1}{\eps^2}\|d_t(|u_h^{m+1}|^2-1)\|_{L^2(\cT_h)}^2 \Bigr)\leq 0. \nonumber
\end{align*}

Finally, applying the summation operator $k\sum_{m=0}^{M-1}$ and
using the definition of $J_\eps^h$ we obtain the desired estimate
\eqref{eq3.12b}. The proof is complete.
\end{proof}

The above theorem immediately infers the following corollary. 

\begin{corollary}\label{cor3.1}
The scheme \eqref{eq3.1}--\eqref{eq3.5} is stable for all $h,k>0$ 
when $f^{m+1}=(u_h^{m+1})^3 - u_h^m$ and  is stable for 
$h>0, 2\eps^2>k>0$ when $f^{m+1}=(u_h^{m+1})^3 - u_h^{m+1}$,
provided that $\sigma_e>\max\{\sigma_0, \sigma_0^{\prime}\}$ for
every $e\in \cE_h$.
\end{corollary}

\begin{theorem}\label{thm3,2}
Suppose that  $\sigma_e>\max\{\sigma_0, \sigma_1\}$ for every $e\in \cE_h$.
Then there exists a unique solution $u_h^{m+1}$ to the scheme 
\eqref{eq3.1}--\eqref{eq3.5} at every time step $t_{m+1}$ for 
$h,k>0$ in the case $f^{m+1}=(u_h^{m+1})^3-u_h^m$.  The 
conclusion still holds provided that $h>0, 2\eps^2>k>0$ in the 
case $f^{m+1}=(u_h^{m+1})^3 - u_h^{m+1}$. 
\end{theorem}

\begin{proof}
Define the following functionals
\begin{align*}
G(v)&:=  k\Phi^h(v) +\frac{k}{\eps^2} \bigl(F(v),1\bigr)_{\cT_h}
 + \frac12 \|v\|_{L^2(\cT_h)}^2 - \bigl(u_h^m, v\bigr)_{\cT_h},\\
H(v)&:= k\Phi^h(v) +\frac{k}{\eps^2} \bigl(F_c^+(v),1\bigr)_{\cT_h}
 + \frac12 \|v\|_{L^2(\cT_h)}^2 
 - \Bigl(\frac{k}{\eps^2}+1\Bigr)\bigl(u_h^m,v\bigr)_{\cT_h}.
\end{align*}
Clearly, $H$ is strictly convex for all $h,k>0$. $G$ is not always convex, however,
it becomes strictly convex when $k<2\eps^2$. To see this, we write $F(v)=F_c^+(v)-F_c^-(v)$
in the definition of $G(v)$ and notice that 
\[
-\frac{k}{\eps^2} \bigl(F_c^-(v),1\bigr)_{\cT_h} + \frac12 \|v\|_{L^2(\cT_h)}^2
=\frac12 \Bigl(1-\frac{k}{\eps^2}\Bigr) \|v\|_{L^2(\cT_h)}^2,
\]
which is strictly convex when $k<2\eps^2$.

Using \eqref{eq3.8}--\eqref{eq3.10}, it is easy to check that problem 
\eqref{eq3.1}--\eqref{eq3.5} is equivalent to the following minimization/variation problems:
\begin{alignat*}{2}
u_h^{m+1}&= \underset{v_h\in V_h}{\mbox{argmin}}\, G(v_h), &&\qquad\mbox{when }
f^{m+1}=(u_h^{m+1})^3-u_h^{m+1},\\
u_h^{m+1}&= \underset{v_h\in V_h}{\mbox{argmin}}\, H(v_h), &&\qquad\mbox{when }
f^{m+1}=(u_h^{m+1})^3-u_h^m.
\end{alignat*}
Thus, the conclusions of the theorem follow from the standard theory
of finite-dimensional convex minimization problems. The proof is complete.
\end{proof}

\subsection{Discrete DG spectrum estimate}\label{sec-3.3}
In this subsection, we shall establish a discrete counterpart of the spectrum 
estimate \eqref{eq2.12} for the DG approximation. Such an estimate 
will play a vital role in our error analysis to be given in the next subsection.
We recall that the desired spectrum estimate was obtained in \cite{Feng_Prohl03} 
for the standard finite element approximation and it plays a vital role 
in the error analysis of \cite{Feng_Prohl03}. Compared with the standard 
finite element approximation, the main additional difficulty for
the DG approximation is caused by the nonconformity of the DG finite 
element space $V_h$ and its mesh-dependent bilinear form $a_h(\cdot, \cdot)$. 

First, we introduce the DG elliptic projection operator $P_r^h: H^s(\cT_h)\to V_h$ by
\begin{equation}\label{eq3.20}
a_h(v-P_r^h v, w_h) + \bigl( v-P_r^h v, w_h \bigr)_{\cT_h} =0
\quad\forall w_h\in V_h
\end{equation}
for any $v\in H^s(\cT_h)$.

Next, we quote the following well known error estimate results from
\cite{Riviere08, Chen_Chen04}.
\begin{lemma}\label{lem3.3}
Let $v\in W^{s,\infty}(\cT_h)$, then there hold 
{\small
\begin{align}\label{eq3.21}
\|v-P_r^h v\|_{L^2(\cT_h)} +h\|\nab(v-P_r^h v)\|_{L^2(\cT_h)} 
&\leq Ch^{\min\{r+1,s\}}\|u\|_{H^s(\cT_h)},\\
\frac{1}{|\ln h|^{\overline{r}}} \|v-P_r^h v\|_{L^\infty(\cT_h)} 
+ h\|\nab(u-P_r^h u)\|_{L^\infty(\cT_h)}
&\leq Ch^{\min\{r+1,s\}}\|u\|_{W^{s,\infty}(\cT_h)}. \label{eq3.22}
\end{align}
}
where $\overline{r}:=\min\{1, r\}-\min\{1, r-1\}$.
\end{lemma}

Let 
\begin{equation}\label{eq3.22a}
C_1=\underset{|\xi|\leq2}{\rm{max}}|f''(\xi)|.
\end{equation}
and $\hP_r^h$, corresponding to $P_r^h$, denote the elliptic projection
operator on the finite element space $S_h:=V_h\cap C^0(\overline{\Ome})$, 
there holds the following estimate \cite{Feng_Prohl03}:
\begin{equation}\label{eq3.22b}
\|u-\hP_r^hu\|_{L^{\infty}}\leq Ch^{2-\frac{d}{2}}||u||_{H^2}.
\end{equation}

We now state our discrete spectrum estimate for the DG approximation.

\begin{proposition}\label{prop3.7}
Suppose there exists a positive number $\gamma>0$ such that the solution $u$ of problem \eqref{eq1.1}--\eqref{eq1.4} satisfies
\begin{equation}\label{eq3.23}
\underset{t\in [0,T]}{\mbox{\rm ess sup}}\, \|u(t)\|_{W^{r+1,\infty}(\Ome)}
\leq C\eps^{-\gamma}.
\end{equation}
Then there exists an $\eps$-independent and $h$-independent constant 
$c_0>0$ such that for $\eps\in(0,1)$ and a.e. $t\in [0,T]$
\begin{align}\label{eq3.24a}
\lambda_h^{\mbox{\tiny DG}}(t):=\inf_{\psi_h\in V_h\atop\psi_h\not\equiv 0}
\frac{ a_h(\psi_h,\psi_h) + \frac{1}{\eps^2}\Bigl( f'\bigl(P_r^h u(t)\bigr)\psi_h, 
\psi_h \Bigr)_{\cT_h}}{\|\psi_h\|_{L^2(\cT_h)}^2} \geq -c_0,
\end{align}
provided that $h$ satisfies the constraint
\begin{align}\label{eq3.24b}
h^{2-\frac{d}{2}}
&\leq C_0 (C_1C_2)^{-1}\eps^{\max\{\sigma_1+3,\sigma_2+2\}},\\
h^{\min\{r+1,s\}}|\ln h|^{\overline{r}} &\leq C_0 (C_1C_2)^{-1}\eps^{\gamma+2},  
\label{eq3.24c} 
\end{align} 
where $C_2$ arises from the following inequality:
\begin{align}\label{eq3.24d}
&\|u-P^h_r u\|_{L^{\infty}((0,T);L^{\infty}(\Ome)}
\leq C_2 h^{\min\{r+1,s\}}|\ln h|^{\overline{r}} \eps^{-\gamma},\\
&\|u-\hP^h_r u\|_{L^{\infty}((0,T);L^{\infty}(\Ome)}
\leq C_2 h^{2-\frac{d}{2}} \eps^{-\max\{\sigma_1+1,\sigma_2\}}. \label{eq3.24e}
\end{align}
\end{proposition}

\begin{proof}
Let $S_h:= V_h\cap C^0(\overline{\Ome})$.  For any $\psi_h\in V_h$, 
we define its finite element (elliptic) projection $\psi_h^{\FE} \in S_h$ by 
\begin{equation}\label{eq3.25}
\widetilde{a}_h(\psi_h^{\FE},\varphi_h) =\widetilde{a}_h(\psi_h,\varphi_h) 
\qquad\forall \varphi_h\in S_h,
\end{equation}
where
\[
\widetilde{a}_h(\psi,\varphi)=a_h(\psi,\varphi) 
+ \beta (\psi,\varphi)_{\cT_h} \qquad\forall \psi,\varphi\in H^2(\cT_h),
\]
and $\beta$ is a positive constant to be specified later. 

By Proposition 8 of \cite{Feng_Prohl03} we have under 
the mesh constraint \eqref{eq3.24b} that
\begin{equation}\label{eq3.26}
\|f'(\hP_r^h u)-f'(u)\|_{L^{\infty}((0,T);L^{\infty}(\Ome))} \leq C_0\eps^2.
\end{equation}
Similarly, under the mesh constraint \eqref{eq3.24c} we can show that
\begin{equation}\label{eq3.27}
\|f'(P^h_r u)-f'(u)\|_{L^{\infty}((0,t);L^{\infty}(\Ome))}
\leq C_0\eps^2.
\end{equation}
Then
\begin{equation}\label{eq3.28}
\|f'(P^h_r u)-f'(\hP^h_r u)\|_{L^{\infty}((0,T);L^{\infty}(\Ome))} \leq 2C_0\eps^2.
\end{equation}
Therefore,
\begin{equation}\label{eq3.29}
f'(P_r^h u) \geq f'(\hP_r^h u)-2C_0\eps^2.
\end{equation}

By the definition of $\psi_h^{\FE}$ we have
\[
a_h(\psi_h,\psi_h)=a_h(\psi_h^{\FE},\psi_h^{\FE}) 
+a_h(\psi_h-\psi_h^{\FE},\psi_h-\psi_h^{\FE})  
-2\beta (\psi_h-\psi_h^{\FE}, \psi_h^{\FE})_{\cT_h}.
\]
Using the above inequality and equality we get
\begin{align}\label{eq3.30}
&a_h(\psi_h,\psi_h) +\frac{1}{\eps^2} \Bigl( f'\bigl(P_r^h u(t)\bigr)\psi_h, 
\psi_h \Bigr)_{\cT_h} \\
%
&\qquad
\geq a_h(\psi_h^{\FE},\psi_h^{\FE})
  +\frac{1}{\eps^2} \Bigl( f'\bigl(\hP_r^h u(t)\bigr),(\psi_h^{\FE})^2\Bigr)_{\cT_h} 
\nonumber\\
&\qquad\qquad
+a_h(\psi_h-\psi_h^{\FE},\psi_h-\psi_h^{\FE})
-2\beta \bigl( \psi_h-\psi_h^{\FE}, \psi_h^{\FE} \bigr)_{\cT_h} \nonumber \\
&\qquad\qquad
+\frac{1}{\eps^2} \Bigl( f'\bigl(\hP_r^h u(t)\bigr), 
(\psi_h)^2-(\psi_h^{\FE})^2 \Bigr)_{\cT_h} -2C_0 \|\psi_h\|_{L^2(\cT_h)}^2.\nonumber 
\end{align}

We now bound the fourth and fifth terms on the right-hand side of \eqref{eq3.30}
from below. For the fourth term we have
\begin{align}\label{eq3.31}
-2\beta \bigl(\psi_h-\psi_h^{\FE}, \psi_h^{\FE} \bigr)_{\cT_h}
&\geq 2\beta\|\psi_h^{\FE}\|_{L^2(\cT_h)}^2 - 2\beta \|\psi_h^{\FE}\|_{L^2(\cT_h)}
\|\psi_h\|_{L^2(\cT_h)} \\
& \geq \beta\|\psi_h^{\FE}\|_{L^2(\cT_h)}^2 - \beta \|\psi_h\|_{L^2(\cT_h)}^2.
\nonumber
\end{align}

To bound the fifth term, by \eqref{eq3.5} and the $L^\infty$-norm estimate
for $u(t)-\hP_r^h u(t)$ we have that under the mesh constraint \eqref{eq3.24b}
\begin{align*}
\|f'(\hP_r^h u(t))\|_{L^\infty(\Ome)} 
&\leq \|f'(u(t))\|_{L^\infty(\Ome)} +  \|f'(u(t))-f'(\hP_r^h u(t))\|_{L^\infty(\Ome)} \\
&\leq \|f'(u(t))\|_{L^\infty(\Ome)} + C\|u(t)-\hP_r^h u(t)\|_{L^\infty(\Ome)}
\leq C.
\end{align*}
Thus, by the algebraic formula $|a^2-b^2|\leq |a-b|^2 + 2|ab|$, we get for some $C>0$
\begin{align}\label{eq3.32}
&\frac{1}{\eps^2} \Bigl( f'\bigl(\hP_r^h u(t)\bigr),
(\psi_h)^2-(\psi_h^{\FE})^2 \Bigr)_{\cT_h}
\geq -\frac{C}{\eps^2} \|(\psi_h)^2-(\psi_h^{\FE})^2\|_{L^1(\cT_h)}\\
&\hskip .8in
\geq -\frac{C}{\eps^2} \Bigl( \|\psi_h-\psi_h^{\FE} \|_{L^2(\cT_h)}^2
+2\|\psi_h-\psi_h^{\FE} \|_{L^2(\cT_h)} \|\psi_h^{\FE} \|_{L^2(\cT_h)}\Bigr)\nonumber\\
&\hskip .8in
\geq -\frac{C}{\eps^2} \Bigl( \bigl(1+\eps^{-2}\bigr) \|\psi_h-\psi_h^{\FE} \|_{L^2(\cT_h)}^2
+\eps^2\|\psi_h^{\FE} \|_{L^2(\cT_h)}^2 \Bigr). \nonumber
\end{align}

Now it comes to a key idea in bounding $ \|\psi_h-\psi_h^{\FE} \|_{L^2(\cT_h)}$,
which is to use the duality argument to bound it from above by 
the energy norm $a_h(\psi_h-\psi_h^{\FE}, \psi_h-\psi_h^{\FE})^{\frac12}$
To the end, we consider the following auxiliary problem: 
find $\phi\in H^1(\Ome)\cap H^2_{\mbox{\tiny loc}}(\Ome)$ such that
\[
\widetilde{a}_h(\phi, \chi) = \bigl(\psi_h-\psi_h^{\FE}, \chi\bigr)_{\cT_h}
\qquad\forall \chi\in H^1(\Ome). 
\]
We assume the above variational problem is $H^{1+\theta}$-regular for some
$\theta\in (0,1]$, that is, there exists a unique 
$\phi\in H^{1+\theta}(\Ome)$ such that
\begin{equation*}
\|\phi\|_{H^{1+\theta}(\Ome)} \leq C \|\psi_h-\psi_h^{\FE}\|_{L^2(\Ome)}.
\end{equation*}
It should be noted that $C(>0)$ can be made independent of $\beta$.

By the definition of $\psi_h^{\FE}$ in \eqref{eq3.25}, we immediately get the 
Galerkin orthogonality
\[
\widetilde{a}_h \bigl(\psi_h-\psi_h^{\FE}, \chi_h\bigr)=0
\qquad\forall \chi_h\in S_h. 
\]
The above orthogonality allows us easily to obtain by the duality argument
(cf. \cite{Riviere08} for a general duality argument for DG methods) 
\begin{equation}\label{eq3.34}
\|\psi_h-\psi_h^{\FE} \|_{L^2(\cT_h)}^2
\leq Ch^{2\theta}\, a_h(\psi_h-\psi_h^{\FE}, \psi_h-\psi_h^{\FE})
\end{equation}
Again, the constant $C$ can be made independent of $\beta$.
 
By Proposition 8 of \cite{Feng_Prohl03} we also have the following 
spectrum estimate in the finite element space $S_h$:
\begin{equation}\label{eq3.35}
a_h(\psi_h^{\FE},\psi_h^{\FE})
  +\frac{1}{\eps^2} \Bigl( f'\bigl(\hP_r^h u(t)\bigr),(\psi_h^{\FE})^2\Bigr)_{\cT_h}
\geq -2C_0 \|\psi_h^{\FE}\|_{L^2(\cT_h)}^2.
\end{equation}

Finally, combining \eqref{eq3.30}--\eqref{eq3.35} we get
\begin{align}\label{eq3.36}
a_h(\psi_h,\psi_h) &+\frac{1}{\eps^2} \Bigl( f'\bigl(P_r^h u(t)\bigr)\psi_h, 
\psi_h \Bigr)_{\cT_h} \\
&\geq \bigl( 1-C h^{2\theta}\eps^{-4}\bigr) a_h(\psi_h-\psi_h^{\FE},\psi_h-\psi_h^{\FE})
\nonumber \\
&\qquad
+\bigl(\beta - C -2C_0 \bigr) \|\psi_h^{\FE}\|_{L^2(\cT_h)}^2
-\bigl(\beta+2C_0\bigr) \|\psi_h\|_{L^2(\cT_h)}^2 \nonumber \\
&\geq -\bigl(\beta+2C_0\bigr) \|\psi_h\|_{L^2(\cT_h)}^2 \qquad\forall \psi_h\in V_h,
\nonumber
\end{align}
provided that $\beta$ is chosen large enough such that 
$\beta- C -2C_0 >0$ and $1-C h^{2\theta}\eps^{-4}>0$, under the mesh constraint
\eqref{eq3.24c}. The proof is complete after setting $c_0=\beta+2C_0$.
\end{proof}

\begin{remark}
The proof actually is constructive in finding the $\eps$- and $h$-independent
constant $c_0$. As expected, $c_0> 2C_0$. We also note that inequality \eqref{eq3.36}
is a G\"arding-type inequality for the non-coercive elliptic operator 
$\mathcal{L}_{\mbox{\tiny AC}}$.
\end{remark}

\subsection{Polynomial order in $\eps^{-1}$ error estimates}\label{sec-3.4}
The goal of this subsection is to derive optimal order error 
estimates for the global error $u(t_m)-u_h^m$ of the fully discrete 
scheme \eqref{eq3.1}--\eqref{eq3.5} under some reasonable mesh
constraints on $h,k$ and regularity assumptions on $u_0$. 
This will be achieved by adapting the nonstandard error estimate
technique with a help of the generalized Gronwall lemma 
(Lemma \ref{lem2.1}) and the discrete spectrum estimate \eqref{eq3.24a}.

The main result of this subsection is the following error 
estimate theorem.

\begin{theorem}\label{thm3.1}
suppose $\sigma_e>\max\{\sigma_0,\sigma_0^{\prime}\}$.
Let $u$ and $\{u_h^m\}_{m=1}^M$ denote respectively the solutions of problems
\eqref{eq1.1}--\eqref{eq1.4} and \eqref{eq3.1}--\eqref{eq3.5}.
Assume $u\in H^2((0,T);$ $L^2(\Ome))\cap L^2((0,T); W^{s,\infty}(\Ome))$ 
and suppose (GA) and \eqref{eq3.23} hold. Then, under the following 
mesh and initial value constraints:

\begin{align*}
h^{2-\frac{d}{2}} &\leq C_0 (C_1C_2)^{-1}\eps^{\max\{\sigma_1+3,\sigma_2+2\}},\\
h^{\min\{r+1,s\}}|\ln h|^{\overline{r}} &\leq C_0 (C_1C_2)^{-1}\eps^{\gamma+2},\\ 
k^2+h^{2\min\{r+1,s\}} &\leq \eps^{4+d+2(\sigma_1+2)}, \\
k &\leq C\eps^{\frac{8+2d+4\sigma_1}{4-d}},\\
u_h^0\in S_h\mbox{ such that }\quad
\|u_0 -u_h^0\|_{L^2(\cT_h)} &\leq C h^{\min\{r+1,s\}},
\end{align*}
there hold 
\begin{align}\label{eq3.36b}
\max_{0\leq m\leq M} \|u(t_m)-u_h^m\|_{L^2(\cT_h)}
&+\Bigl( k^2\sum_{m=1}^M\| d_{t}(u(t_m)-u_h^m)\|_{L^2(\cT_h)}^2\Bigr)^{\frac{1}{2}}\\
&\leq C(k+h^{\min\{r+1,s\}})\eps^{-(\sigma_1+2)}.\nonumber\\
\Bigl( k\sum_{m=1}^M \| u(t_m)-u_h^m\|_{H^1(\cT_h)}^2 \Bigr)^{\frac{1}{2}}
&\leq C(k+h^{\min\{r+1,s\}-1})\eps^{-(\sigma_1+3)}, \label{eq3.36c} \\ 
\max_{0\leq m\leq M} \|u(t_m)-u_h^m\|_{L^\infty(\cT_h)}
&\leq C h^{\min\{r+1,s\}} |\ln h|^{\overline{r}} \eps^{-\gamma} \label{eq3.36d} \\
&\qquad
+Ch^{-\frac{d}{2}}(k+h^{\min\{r+1,s\}})\eps^{-(\sigma_1+2)}.  \nonumber
\end{align}
\end{theorem}

\begin{proof}
We only give a proof for the case $f^{m+1}=(u_h^{m+1})^3-u_h^m$ because 
its proof is slightly more difficult than that for the case 
$f^{m+1}=(u_h^{m+1})^3-u_h^{m+1}$. Since the proof is long, we divide it into four steps.

\medskip
{\em Step 1}: 
We begin with introducing the following error decompositions:
\[
u(t_m)-u_h^m=\eta^m + \xi^m,\quad \eta^m:=u(t_m)-P_r^h u(t_m),\quad
\xi^m:=P_r^h u(t_m)- u_h^m.
\]
It is easy to check that the exact solution $u$ satisfies 
\begin{align}\label{eq3.37}
\bigl(d_t u(t_{m+1}),v_h\bigr)_{\cT_h} +a_h(u(t_{m+1}),v_h)
+\frac{1}{\eps^2}\bigl( f(u(t_{m+1})), v_h\bigr)_{\cT_h} 
= \bigl( R_{m+1}, v_h\bigr)_{\cT_h} 
\end{align}
for all $ v_h\in V_h$, where
\[
R_{m+1}:=-\frac{1}{k}\int_{t_m}^{t_{m+1}}(t-t_m)u_{tt}(t)\,dt.
\]
Hence
\begin{align}\label{eq3.38}
k\sum_{m=0}^{\ell}\|R_{m+1}\|_{L^2(\Ome)}^2
&\leq \frac{1}{k}\mathop{\sum}\limits_{m=0}^{\ell}
\Bigl(\int_{t_m}^{t_{m+1}}(s-t_m)^2ds\Bigr)
\Bigl(\int_{t_m}^{t_{m+1}}\|u_{tt}\|_{L^2(\Ome)}^2\,ds\Bigr) \\
&\leq Ck^2\eps^{-2\max\{\sigma_1+2,\sigma_3\}}. \nonumber
\end{align}

Subtracting \eqref{eq3.1} from \eqref{eq3.37} and using the definitions
of $\eta^m$ and $\xi^m$ we get the following error equation:
\begin{align}\label{eq3.39}
&\bigl(d_t \xi^{m+1},v_h\bigr)_{\cT_h} +a_h(\xi^{m+1},v_h)
+\frac{1}{\eps^2}\bigl( f(u(t_{m+1}))-f^{m+1}, v_h\bigr)_{\cT_h} \\
&\hskip 0.5in 
= \bigl( R_{m+1}, v_h\bigr)_{\cT_h} 
-\bigl(d_t \eta^{m+1},v_h\bigr)_{\cT_h}-a_h(\eta^{m+1},v_h) \nonumber\\
&\hskip 0.5in 
= \bigl( R_{m+1}, v_h\bigr)_{\cT_h} 
-\bigl(d_t \eta^{m+1},v_h\bigr)_{\cT_h} +(\eta^{m+1},v_h)_{\cT_h}. \nonumber
\end{align}
Setting $v_h=\xi^{m+1}$ and using Schwarz inequality yield 
\begin{align*}
&\frac12 \Bigl( d_t \|\xi^{m+1}\|_{L^2(\cT_h)}^2 + k\|d_t\xi^{m+1}\|_{L^2(\cT_h)}^2\Bigr)
+a_h(\xi^{m+1},\xi^{m+1}) \\
&\hskip 1.15in
+\frac{1}{\eps^2}\bigl(f(u(t_{m+1}))-f^{m+1},\xi^{m+1}\bigr)_{\cT_h} \nonumber \\
&\qquad
\leq \Bigl( \|R_{m+1}\|_{L^2(\cT_h)}  
+\|d_t \eta^{m+1}\|_{L^2(\cT_h)} +\|\eta^{m+1}\|_{L^2(\cT_h)} \Bigr) 
\|\xi^{m+1}\|_{L^2(\cT_h)}. \nonumber
\end{align*}
Summing in $m$ (after having lowered the index by $1$) from 
$1$ to $\ell\, (\leq M)$ and using \eqref{eq3.21} and \eqref{eq3.38} we get
\begin{align}\label{eq3.40}
&\|\xi^\ell\|_{L^2(\cT_h)}^2 + 2k\sum_{m=1}^\ell k\|d_t\xi^m\|_{L^2(\cT_h)}^2 
+ 2k\sum_{m=1}^\ell a_h(\xi^m,\xi^m) \\
&\hskip 0.7in
+2k\sum_{m=1}^\ell \frac{1}{\eps^2}\bigl(f(u(t_m))-f^m,\xi^m\bigr)_{\cT_h} \nonumber \\
&\hskip 0.3in
\leq k\sum_{m=1}^\ell \|\xi^m\|_{L^2(\cT_h)}^2 +2\|\xi^0\|_{L^2(\cT_h)}^2
+ C\Bigl( k^2\eps^{-2\max\{\sigma_1+2,\sigma_3\}}  \nonumber\\
&\hskip 0.7in
+h^{2\min\{r+1,s\}}\, \|u\|_{H^1((0,T);H^s(\Ome))}^2 \Bigr). \nonumber
\end{align}

\medskip
{\em Step 2}:  We now bound the fourth term on the left-hand side of 
\eqref{eq3.40}. By the definition of $f^m$ we have
\begin{align*}
&f(u(t_m))-f^m = f(u(t_m)) -f\bigl(P_r^h u(t_m)\bigr) + f\bigl(P_r^h u(t_m)\bigr)
-f^m\\ 
&\qquad\geq -\bigl|f(u(t_m)) -f\bigl(P_r^h u(t_m)\bigr) \bigr| 
+ \bigl(P_r^h u(t_m)\bigr)^3 -P_r^h u(t_m) - (u_h^m)^3 + u_h^{m-1} \nonumber \\
&\qquad\geq -C|\eta^m| + \Bigl( \bigl(P_r^h u(t_m)\bigr)^2 + P_r^h u(t_m)\,
u_h^m + (u_h^m)^2 \Bigr)\xi^m -\xi^m -kd_t u_h^m \nonumber \\
&\qquad\geq -C|\eta^m| +\Bigl( 3\bigl(P_r^h u(t_m)\bigr)^2-1 \Bigr)\xi^m
 - 3P_r^h u(t_m)\, (\xi^m)^2 + (\xi^m)^3 -kd_t u_h^m \nonumber \\
&\qquad\geq -C|\eta^m| +f'\bigl(P_r^h u(t_m)\bigr)\,\xi^m
 - 3P_r^h u(t_m)\, (\xi^m)^2 + (\xi^m)^3 -kd_t u_h^m. \nonumber
\end{align*}
Hence
\begin{align*}
&2k\sum_{m=1}^\ell \frac{1}{\eps^2}\bigl(f(u(t_m))-f^m,\xi^m\bigr)_{\cT_h}\\
&\geq -\frac{Ck}{\eps^2}\sum_{m=1}^\ell \||\eta^m\|_{L^2(\cT_h)} \|\xi^m\|_{L^2(\cT_h)}
+2k\sum_{m=1}^\ell \frac{1}{\eps^2}\Bigl( f'\bigl(P_r^h u(t_m)\bigr), (\xi^m)^2
\Bigr)_{\cT_h} \nonumber \\
&\quad 
-\frac{Ck}{\eps^2} \sum_{m=1}^\ell \|\xi^m\|_{L^3(\cT_h)}^3 
+\frac{2k}{\eps^2}\sum_{m=1}^\ell \|\xi^m\|_{L^4(\cT_h)}^4 
-\frac{2k}{\eps^2} \sum_{m=1}^\ell k\|d_t u_h^m\|_{L^2(\cT_h)}\,\|\xi^m\|_{L^2(\cT_h)} \nonumber \\
&\geq 2k\sum_{m=1}^\ell \frac{1}{\eps^2}\Bigl( f'\bigl(P_r^h u(t_m)\bigr), 
(\xi^m)^2 \Bigr)_{\cT_h}  
+\frac{2k}{\eps^2}\sum_{m=1}^\ell \|\xi^m\|_{L^4(\cT_h)}^4
-\frac{Ck}{\eps^2} \sum_{m=1}^\ell \|\xi^m\|_{L^3(\cT_h)}^3 \nonumber\\
&\quad 
-k\sum_{m=1}^\ell \|\xi^m\|_{L^2(\cT_h)}^2 
-C\Bigl( h^{2\min\{r+1,s\}} \eps^{-4} \|u\|_{L^2((0,T);H^s(\Ome)}^2
+k^2 \eps^{-4} J_\eps^h(u_h^0) \Bigr).\nonumber
\end{align*}
Here we have used  the fact that $|P_r^h u(t_m)|\leq C$ and \eqref{eq3.12b}.

Substituting the above estimate into \eqref{eq3.40} yields
\begin{align}\label{eq3.41}
&\|\xi^\ell\|_{L^2(\cT_h)}^2 + 2k\sum_{m=1}^\ell k\|d_t\xi^m\|_{L^2(\cT_h)}^2 
+ \frac{2}{\eps^2} k\sum_{m=1}^\ell \|\xi^m\|_{L^4(\cT_h)}^4 \\
&\hskip 0.65in
+ 2k\sum_{m=1}^\ell \left(  a_h(\xi^m,\xi^m) 
+\frac{1}{\eps^2} \Bigl(f'\bigl(P_r^h u(t_m)\bigr),(\xi^m)^2\Bigr)_{\cT_h}\right) 
\nonumber \\
&\hskip 0.1in
\leq 2k\sum_{m=1}^\ell \|\xi^m\|_{L^2(\cT_h)}^2 
+ \frac{Ck}{\eps^2} \sum_{m=1}^\ell \|\xi^m\|_{L^3(\cT_h)}^3 \nonumber\\
&\hskip 0.65in
+2\|\xi^0\|_{L^2(\cT_h)}^2
+ Ck^2\Bigl( \eps^{-2\max\{\sigma_1+2,\sigma_3\}}+\eps^{-4} J_\eps^h(u_h^0)
\Bigr) \nonumber\\
&\hskip 0.65in
+C h^{2\min\{r+1,s\}} \Bigl(\|u\|_{H^1((0,T);H^s(\Ome))}^2 
+ \eps^{-4} \|u\|_{L^2((0,T);H^s(\Ome)}^2 \Bigr). \nonumber
\end{align}

\medskip
{\em Step 3}: To control the second term on the right-hand
side of \eqref{eq3.41} we use the following Gagliardo-Nirenberg
inequality \cite{Adams03}:
\[
\|v\|_{L^3(K)}^3\leq C\Bigl( \|\nab v\|_{L^2(K)}^{\frac{d}2} 
\bigl\|v\bigr\|_{L^2(K)}^{\frac{6-d}2} +\|v\|_{L^2(K)}^3 \Bigr)
\qquad\forall K\in \cT_h
\]
to get
\begin{align}\label{eq3.42}
\frac{Ck}{\eps^2} \sum_{m=1}^\ell \|\xi^m\|_{L^3(\cT_h)}^3
&\leq \eps^2\alpha k\sum_{m=1}^\ell \|\nab \xi^m\|_{L^2(\cT_h)}^2
+\eps^2 k\sum_{m=1}^\ell \|\xi^m\|_{L^2(\cT_h)}^2 \\
&\qquad
+C\eps^{-\frac{2(4+d)}{4-d}} k\sum_{m=1}^\ell\sum_{K\in \cT_h}
\bigl\|\xi^m\bigr\|_{L^2(K)}^{\frac{2(6-d)}{4-d}} \nonumber \\
&\leq \eps^2\alpha  k\sum_{m=1}^\ell \|\nab \xi^m\|_{L^2(\cT_h)}^2
+\eps^2 k\sum_{m=1}^\ell \|\xi^m\|_{L^2(\cT_h)}^2 \nonumber \\
&\qquad
+C\eps^{-\frac{2(4+d)}{4-d}} k\sum_{m=1}^\ell
\bigl\|\xi^m\bigr\|_{L^2(\cT_h)}^{\frac{2(6-d)}{4-d}}. \nonumber
\end{align}

Finally, for the fourth term on the left-hand side of  \eqref{eq3.41}
we utilize the discrete spectrum estimate \eqref{eq3.24a} to
bound it from below as follows:
\begin{align}\label{eq3.43}
&2k\sum_{m=1}^\ell \left( a_h(\xi^m,\xi^m) 
+\frac{1}{\eps^2} \Bigl(f'\bigl(P_r^h u(t_m)\bigr),(\xi^m)^2\Bigr)_{\cT_h}\right)
\\
&= 2(1-\eps^2) k\sum_{m=1}^\ell \left( a_h(\xi^m,\xi^m) 
+\frac{1}{\eps^2} \Bigl(f'\bigl(P_r^h u(t_m)\bigr),(\xi^m)^2\Bigr)_{\cT_h}\right)
\nonumber\\
&\quad\qquad
+2\eps^2 k\sum_{m=1}^\ell \left( a_h(\xi^m,\xi^m) 
+\frac{1}{\eps^2} \Bigl(f'\bigl(P_r^h u(t_m)\bigr),(\xi^m)^2\Bigr)_{\cT_h}\right)
\nonumber \\ 
&\geq -2(1-\eps^2)c_0 k\sum_{m=1}^\ell \|\xi^m\|_{L^2(\cT_h)}^2 
+ 4\eps^2 \alpha k\sum_{m=1}^\ell \|\xi\|_{1,\DG}^2 
-Ck \sum_{m=1}^\ell \|\xi^m\|_{L^2(\cT_h)}^2,
\nonumber
\end{align}
where we have used \eqref{eq3.12} and \eqref{eq3.11a} to get the
second term on the right-hand side.

\medskip
{\em Step 4}: Substituting \eqref{eq3.42} and \eqref{eq3.43} 
into \eqref{eq3.41} we get
\begin{align}\label{eq3.44}
&\|\xi^\ell\|_{L^2(\cT_h)}^2 + k\sum_{m=1}^\ell 
\Bigl( 2k\|d_t\xi^m\|_{L^2(\cT_h)}^2 
+ 3\eps^2 \alpha \|\xi\|_{1,\DG}^2 \Bigr) \\
&\hskip 0.3in
\leq C(1+c_0) k\sum_{m=1}^\ell \|\xi^m\|_{L^2(\cT_h)}^2 
+C\eps^{-\frac{2(4+d)}{4-d}} k\sum_{m=1}^\ell
\bigl\|\xi^m\bigr\|_{L^2(\cT_h)}^{\frac{2(6-d)}{4-d}} \nonumber\\
&\hskip 0.6in
+2\|\xi^0\|_{L^2(\cT_h)}^2
+ Ck^2\Bigl( \eps^{-2\max\{\sigma_1+2,\sigma_3\}}+\eps^{-4} J_\eps^h(u_h^0)
\Bigr) \nonumber\\
&\hskip 0.6in
+C h^{2\min\{r+1,s\}} \Bigl(\|u\|_{H^1((0,T);H^s(\Ome))}^2 
+ \eps^{-4} \|u\|_{L^2((0,T);H^s(\Ome)}^2 \Bigr). \nonumber
\end{align}

At this point, notice that there are
two terms on the right-hand side of \eqref{eq3.44} that involve
the approximated initial datum $u_h^0$. On one hand, we need to choose
$u_h^0$ such that $\|\xi^0\|_{L^2(\cT_h)}=O(h^{\min\{r+1,s\}})$ 
to maintain the optimal rate of convergence in $h$. Clearly, 
both the $L^2$ and the elliptic projection of $u_0$ will work.
In fact, in the latter case, $\xi^0=0$.	On the other hand, 
we want $J_\eps^h(u_h^0)$ to be uniformly bounded in $h$.
but the jump term in $J_\eps^h(u_h^0)$ always depend
on $h$ unless it vanishes. To satisfy this requirement,
we ask $u_h^0\in S_h$. Therefore, we are led to choose $u_h^0$ 
to be the $L^2$ or the elliptic projection of $u_0$
into the finite element space $S_h$.

It then follows from \eqref{eq2.5b}, \eqref{eq2.7}, \eqref{eq2.10} and \eqref{eq3.44} that
\begin{align}\label{eq3.45}
&\|\xi^\ell\|_{L^2(\cT_h)}^2 + k\sum_{m=1}^\ell 
\Bigl( 2k\|d_t\xi^m\|_{L^2(\cT_h)}^2 
+ 3\eps^2 \alpha \|\xi\|_{1,\DG}^2 \Bigr) \\
&\hskip 0.3in
\leq C(1+c_0)k\sum_{m=1}^\ell \|\xi^m\|_{L^2(\cT_h)}^2 
+C\eps^{-\frac{2(4+d)}{4-d}} k\sum_{m=1}^\ell
\bigl\|\xi^m\bigr\|_{L^2(\cT_h)}^{\frac{2(6-d)}{4-d}} \nonumber\\
&\hskip 0.6in
+ Ck^2\eps^{-2(\sigma_1+2)}+C h^{2\min\{r+1,s\}}\eps^{-2(\sigma_1+2)}. \nonumber
\end{align}
 
On noting that $u_h^{\ell}$ can be written as
\begin{equation}\label{eq3.46}
u^{\ell}_h=k\mathop{\sum}\limits_{m=1}^\ell d_tu^m_h+u^0_h,
\end{equation}
then by \eqref{eq2.2} and \eqref{eq3.12b}, we get
\begin{align}\label{eq3.47}
\|u^{\ell}_h\|_{L^2(\mathcal{T}_h)}
\leq k\mathop{\sum}\limits_{m=1}^\ell \|d_tu^m_h\|_{L^2(\mathcal{T}_h)}
+\|u^0_h\|_{L^2(\mathcal{T}_h)}
\leq C\eps^{-2\sigma_1}.
\end{align}
By the boundedness of the projection, we have
\begin{equation}\label{eq3.48}
\|\xi^\ell\|_{L^2(\mathcal{T}_h)}^2\leq C\eps^{-2\sigma_1}.
\end{equation}
Then \eqref{eq3.45} can be reduced to
\begin{equation}\label{eq3.49}
\|\xi^\ell\|_{L^2(\mathcal{T}_h)}^2 + k\sum_{m=1}^\ell 
\Bigl( 2k\|d_t\xi^m\|_{L^2(\cT_h)}^2 
+ 3\eps^2 \alpha \|\xi\|_{1,\DG}^2 \Bigr) \leq M_1+M_2, 
\end{equation} 
where
\begin{align}\label{eq3.50}
M_1:&=C(1+c_0) k\sum_{m=1}^{\ell-1} \|\xi^m\|_{L^2(\cT_h)}^2 
+ C\eps^{-\frac{2(4+d)}{4-d}} k\sum_{m=1}^{\ell-1}
\bigl\|\xi^m\bigr\|_{L^2(\cT_h)}^{\frac{2(6-d)}{4-d}}\\ \nonumber
&\qquad\qquad
+Ck^2\eps^{-2(\sigma_1+2)}+C h^{2\min\{r+1,s\}}\eps^{-2(\sigma_1+2)},\nonumber\\
\label{eq3.51}
M_2:&=C(1+c_0) k\|\xi^{\ell}\|_{L^2(\cT_h)}^2 + C\eps^{-\frac{2(4+d)}{4-d}} k
\bigl\|\xi^{\ell}\bigr\|_{L^2(\cT_h)}^{\frac{2(6-d)}{4-d}}.
\end{align}

It is easy to check that
\begin{equation}\label{eq3.52}
M_2<\frac12\|\xi^\ell\|_{L^2(\mathcal{T}_h)}^2 
\qquad\mbox{provided that}\quad k<C\eps^{\frac{8+2d+4\sigma_1}{4-d}}.
\end{equation}
By \eqref{eq3.49} we have
\begin{align}
& \|\xi^\ell\|_{L^2(\mathcal{T}_h)}^2 + k\sum_{m=1}^\ell 
\Bigl( 2k\|d_t\xi^m\|_{L^2(\cT_h)}^2 
+ 3\eps^2 \alpha \|\xi\|_{1,\DG}^2 \Bigr)\leq2M_1\label{eq3.53}\\
&\qquad
=2C(1+c_0) k\sum_{m=1}^{\ell-1} \|\xi^m\|_{L^2(\cT_h)}^2 
+2C\eps^{-\frac{2(4+d)}{4-d}} k\sum_{m=1}^{\ell-1}
\bigl\|\xi^m\bigr\|_{L^2(\cT_h)}^{\frac{2(6-d)}{4-d}} \nonumber\\
&\qquad\qquad
+2Ck^2\eps^{-2(\sigma_1+2)}+C h^{2\min\{r+1,s\}}\eps^{-2(\sigma_1+2)}\nonumber\\
&\qquad
\leq C(1+c_0) k\sum_{m=1}^{\ell-1} \|\xi^m\|_{L^2(\cT_h)}^2 + C\eps^{-\frac{2(4+d)}{4-d}} k\sum_{m=1}^{\ell-1}
\bigl\|\xi^m\bigr\|_{L^2(\cT_h)}^{\frac{2(6-d)}{4-d}} \nonumber\\
&\qquad\qquad+ Ck^2\eps^{-2(\sigma_1+2)}+C h^{2\min\{r+1,s\}}\eps^{-2(\sigma_1+2)}.\nonumber 
\end{align}
 
Let $d_{\ell}\geq 0$ be the slack variable such that
\begin{align}
& \|\xi^\ell\|_{L^2(\mathcal{T}_h)}^2 + k\sum_{m=1}^\ell 
\Bigl( 2k\|d_t\xi^m\|_{L^2(\cT_h)}^2 
+ 3\eps^2 \alpha \|\xi\|_{1,\DG}^2 \Bigr)+d_{\ell}\label{eq3.54}\\
&\qquad
=C(1+c_0) k\sum_{m=1}^{\ell-1} \|\xi^m\|_{L^2(\cT_h)}^2 + C\eps^{-\frac{2(4+d)}{4-d}} k\sum_{m=1}^{\ell-1}
\bigl\|\xi^m\bigr\|_{L^2(\cT_h)}^{\frac{2(6-d)}{4-d}} \nonumber\\
&\qquad\qquad+ Ck^2\eps^{-2(\sigma_1+2)}+C h^{2\min\{r+1,s\}}\eps^{-2(\sigma_1+2)},\nonumber
\end{align}
and define for $\ell\geq1$
\begin{align}\label{eq3.55}
S_{\ell+1}:&=  \|\xi^\ell\|_{L^2(\mathcal{T}_h)}^2 + k\sum_{m=1}^\ell 
\Bigl( 2k\|d_t\xi^m\|_{L^2(\cT_h)}^2 + 3\eps^2 \alpha \|\xi\|_{1,\DG}^2 \Bigr)+d_{\ell},\\
\label{eq3.57}
S_{1}:&=Ck^2\eps^{-2(\sigma_1+2)}+C h^{2\min\{r+1,s\}}\eps^{-2(\sigma_1+2)},
\end{align}
then we have
\begin{equation}\label{eq3.56}
S_{\ell+1}-S_{\ell}\leq C(1+c_0) kS_{\ell}+C\eps^{-\frac{2(4+d)}{4-d}} kS_{\ell}^{\frac{6-d}{4-d}}\qquad\text{for}\ \ell\geq1.
\end{equation}
Applying Lemma \ref{lem2.1} to $\{S_\ell\}_{\ell\geq 1}$ defined above, 
we obtain for $\ell\geq1$
\begin{equation}\label{eq3.58}
S_{\ell}\leq a^{-1}_{\ell}\Bigg\{S^{-\frac{2}{4-d}}_{1}-\frac{2Ck}{4-d}
\sum_{s=1}^{\ell-1}\eps^{-\frac{2(4+d)}{4-d}} a^{-\frac{2}{4-d}}_{s+1}\Bigg\}^{-\frac{4-d}{2}}
\end{equation}
provided that
\begin{equation}\label{eq3.59}
\frac12 S^{-\frac{2}{4-d}}_{1}-\frac{2Ck}{4-d}\sum_{s=1}^{\ell-1} \eps^{-\frac{2(4+d)}{4-d}} 
a^{-\frac{2}{4-d}}_{s+1}>0.
\end{equation}
We note that $a_s\, (1\leq s\leq \ell)$ are all bounded as $k\rightarrow0$, 
therefore, \eqref{eq3.59} holds under the mesh constraint stated in the theorem.  
It follows from \eqref{eq3.58} and \eqref{eq3.59} that 
\begin{equation}\label{eq3.61}
S_{\ell}\leq 2a_\ell^{-1} S_1
\leq Ck^2\eps^{-2(\sigma_1+2)}+C h^{2\min\{r+1,s\}}\eps^{-2(\sigma_1+2)}.
\end{equation}

Finally, using the above estimate and the properties of the operator $P^h_r$ 
we obtain \eqref{eq3.36b} and \eqref{eq3.36c}. The estimate \eqref{eq3.36d} follows
from \eqref{eq3.36c} and the inverse inequality bounding the $L^\infty$-norm by the
$L^2$-norm and \eqref{eq3.24d}.  The proof is complete. 
\end{proof}

\section{Convergence of the numerical interface to the mean curvature flow} \label{sec-4}
In this section, we establish the convergence and rate of convergence of 
the numerical interface $\Gamma_t^{\eps,h,k}$, which is defined as the zero-level
set of the numerical solution $\{u_h^n\}$ (see the precise definition below),
to the sharp interface limit (the mean curvature flow) of the Allen-Cahn 
equation. The key ingredient of the proof is the $L^{\infty}(J; L^{\infty})$ 
error estimate obtained in the previous section, which depends on $\eps^{-1}$ in 
a low polynomial order. It is proved that the numerical interface converges with the rate
$O(\epsilon^2|\ln\,\epsilon|^2)$ before the singularities appear, 
and with the rate $O(\epsilon)$ provided the mean curvature flow does
not develop an interior. It should be noted that the proof to be given below essentially 
follows the same lines as in the proof of \cite{Feng_Prohl03}. For the reader's convenience, 
we provide here a self-contained proof. Throughout this section, $u^\epsilon$ denotes
the solution of the Allen-Cahn problem \eqref{eq1.1}--\eqref{eq1.4}. 

We notice that, unlike in the finite element case, the DG solution $u_h^n$ is discontinuous
in space (and in time).  As a result, the zero-level set of $u_h^n$ may not be well defined. 
To circumvent this technicality, we introduce the finite element approximation 
$\widehat{u}_h^m$ of $u_h^m$ which is defined using the averaged degrees of freedom 
of $u_h^n$ as the degrees of freedom for determining $\widehat{u}_h^m$ (cf. \cite{KP04}).
The following approximation result was proved in Theorem 2.1 in \cite{KP04}.

\begin{theorem}\label{lem4.1} 
Let $\mathcal{T}_h$ be a conforming mesh consisting of 
triangles when $d=2$, and tetrahedra when $d=3$. For $v_h\in V_h$, let 
$\widehat{v}_h$ be the finite element approximation  of $v_h$ as
defined above. Then for any $v_h\in V_h$ and $i=0,1$ there holds
\begin{align}\label{eqn_KP}
\sum_{K\in\mathcal{T}_h} \|v_h-\widehat{v}_h\|_{H^i(K)}^2
\leq C \sum_{e\in\cE_h^I} h^{1-2i}_e \|[v_h]\|_{L^2(e)}^2,
\end{align}
where $C>0$ is a constant independent of $h$ and $v_h$ but may depend on $r$ and the minimal 
angle $\theta_0$ of the triangles in $\mathcal{T}_h$. 
\end{theorem}

Using the above approximation result we can show that the error estimates 
of Theorem \ref{thm3.1} also hold for $\widehat{u}_h^n$.

\begin{theorem}\label{lem4.2}
Let $u_h^{m}$ denote the solution of the DG scheme \eqref{eq3.1}--\eqref{eq3.5}
and $\widehat{u}_h^{m}$ denote its finite element approximation as defined above. Then 
under the assumptions of Theorem \ref{thm3.1} the error estimates for $u_h^m$ given in 
Theorem \ref{thm3.1} are still valid for $\widehat{u}_h^{m}$, in particular, there holds
\begin{align}\label{eq3.36bx}
\max_{0\leq m\leq M} \|u(t_m)-\widehat{u}_h^m\|_{L^\infty(\cT_h)}
&\leq C h^{\min\{r+1,s\}} |\ln h|^{\overline{r}} \eps^{-\gamma}\\
&\qquad
+Ch^{-\frac{d}{2}}(k+h^{\min\{r+1,s\}})\eps^{-(\sigma_1+2)}.  \nonumber
\end{align}

\end{theorem}

\begin{proof}
We only give a proof for \eqref{eq3.36bx} because other estimates can 
be proved likewise.  By the triangle inequality we have
\begin{align}\label{eq3.36by}
\|u(t_m)-\widehat{u}_h^m\|_{L^\infty(\cT_h)}
\leq \|u(t_m)-u_h^m\|_{L^\infty(\cT_h)}+ \|u_h^m-\widehat{u}_h^m\|_{L^\infty(\cT_h)}.
\end{align}
Hence, it suffices to show that the second term on the right-hand side 
is an equal or higher order term compared to the first one. 

Let $u^I(t)$ denote the finite element interpolation of $u(t)$ into $S_h$.  
It follows from \eqref{eqn_KP} and the trace inequality that
\begin{align}\label{eq3.36bz}
\|u_h^m-\widehat{u}_h^m\|_{L^2(\cT_h)}^2 
&\leq C\sum_{e\in \cE_h^I} h_e \|[u_h^m]\|_{L^2(e)}^2 \\
&= C\sum_{e\in \cE_h^I} h_e \|[u_h^m-u^I(t_m)]\|_{L^2(e)}^2 \nonumber\\
& \leq C\sum_{K\in \cT_h} h_e h_K^{-1}\|u_h^m-u^I(t_m)\|_{L^2(K)}^2 \nonumber \\
&\leq C \bigl( \|u_h^m-u(t_m)\|_{L^2(\cT_h)}^2 + \|u(t_m)- u^I(t_m)\|_{L^2(\cT_h)}^2 \bigr). \nonumber
\end{align}
Substituting \eqref{eq3.36bz} into \eqref{eq3.36by} after using the inverse inequality yields
\begin{align*}
&\|u(t_m)-\widehat{u}_h^m\|_{L^\infty(\cT_h)}
\leq \|u(t_m)-u_h^m\|_{L^\infty(\cT_h)}+ C h^{-\frac{d}2} \|u_h^m-\widehat{u}_h^m\|_{L^2(\cT_h)}\\
&\qquad\quad
\leq \|u(t_m)-u_h^m\|_{L^\infty(\cT_h)}  \nonumber \\
&\qquad\qquad
+ Ch^{-\frac{d}2}  \bigl( \|u_h^m-u(t_m)\|_{L^2(\cT_h)} + \|u(t_m)- u^I(t_m)\|_{L^2(\cT_h)} \bigr),\nonumber 
\end{align*}
which together with \eqref{eq3.36b} implies the desired estimate \eqref{eq3.36bx}. The proof is complete.
\end{proof}

We are now ready to state the main theorem of this section.

\begin{theorem}\label{thm4.2}
Let $\{\Gamma_t\}$ denote the (generalized) mean curvature flow 
defined in \cite{ESS92}, that is, $\Gamma_t$ is the zero-level set of 
the solution $w$ of the following initial value problem:
\begin{alignat}{2}\label{eq4.1}
w_t &=\Delta w-\frac{D^2wDw\cdot Dw}{|Dw|^2} &&\qquad\mbox{in } 
\mathbf{R}^d\times (0,\infty)\\
w(\cdot,0) &=w_0(\cdot) &&\qquad\mbox{in } \mathbf{R}^d. \label{eq4.2}
\end{alignat}
Let $u^{\epsilon,h,k}$ denote the piecewise linear interpolation in time 
of the numerical solution $\{\widehat{u}_h^m\}$ defined by
\begin{equation}\label{eq4.3}
u^{\epsilon,h,k}(x,t):=\frac{t-t_m}{k}\widehat{u}_h^{m+1}(x)+\frac{t_{m+1}-t}{k} 
\widehat{u}_h^{m}(x), \quad t_m\leq t\leq t_{m+1}
\end{equation}
for $0\leq m\leq M-1$. Let $\{\Gamma_t^{\epsilon,h,k}\}$ denote the zero-level 
set of $u^{\epsilon,h,k}$, namely,
\begin{equation}\label{eq4.4}
\Gamma_t^{\epsilon,h,k}=\{x\in \Omega;\, u^{\epsilon,h,k}(x,t)=0\}.
\end{equation}
Suppose $\Gamma_0=\{x\in \overline{\Omega};u_0(x)=0\}$ is a smooth hypersurface 
compactly contained in $\Omega$, and $k=O(h^2)$. Let $t_*$ be
the first time at which the mean curvature flow develops a singularity, then we have 
\begin{itemize}
\item[(i)] suppose $\{\Gamma_t\}_{t\geq 0}$ does not develop an interior, 
then there exists a constant $\epsilon_1>0$ such that for all
$\epsilon\in(0,\epsilon_1)$ and  $t> 0$ there holds
\[
\sup_{x\in\Gamma_t^{\epsilon,h,k}}\{\mbox{\rm dist}(x,\Gamma_t)\}
\leq C\epsilon;
\] 
\item[(ii)] there exists a constant $\epsilon_2>0$ such that for all 
$\epsilon\in(0,\epsilon_2)$ and  $ 0<t<t_*$ there holds
\[
\sup_{x\in\Gamma_t^{\epsilon,h,k}}\{\mbox{\rm dist}(x,\Gamma_t)\}
\leq C\epsilon^2|\ln\,\epsilon|^2.
\] 
\end{itemize}
\end{theorem}

\begin{proof}
We note that since $u^{\epsilon,h,k}(x,t)$ is continuous in both $t$ and $x$, then 
$\Gamma_t^{\epsilon,h,k}$ is well defined.

\medskip
{\em Step 1}:  Let $I_t$ and $O_t$ denote the inside and the outside of $\Gamma_t$ defined by
\begin{equation}\label{eq4.5}
I_t:=\{x\in \mathbf{R}^d; w(x,t)>0\},\quad 
O_t:=\{x\in \mathbf{R}^d;w(x,t)<0\}.
\end{equation}
Let $d(x,t)$ denote the signed distance function to $\Gamma_t$ which is positive 
in $I_t$ and negative in $O_t$. By Theorem 4.1 of \cite{ESS92}, there exist 
$\widehat{\epsilon}_1>0$ and $\widehat{C}_1>0$ such that for all $t\geq  0$ 
and $\epsilon\in(0,\widehat{\epsilon}_1)$ there hold
\begin{alignat}{2}\label{eq4.6}
u^{\epsilon}(x,t) &\geq 1-\epsilon  
&&\qquad\forall x\in\{x\in\overline{\Omega};\,d(x,t)\geq \widehat{C}_1\epsilon\},\\
u^{\epsilon}(x,t) &\leq-1+\epsilon
&&\qquad\forall x\in\{x\in\overline{\Omega};\,d(x,t)\leq -\widehat{C}_1\epsilon\}.
\label{eq4.7}
\end{alignat}

Since for any fixed $x\in\Gamma_t^{\epsilon,h,k}$, $u^{\epsilon,h,k}(x,t)=0$, 
by \eqref{eq3.36bx} with $k=O(h^2)$, we have 
\begin{align*}
|u^{\epsilon}(x,t)| &=|u^{\epsilon}(x,t)-u^{\epsilon,h,k}(x,t)| \\
&\leq \tilde{C} \Bigl( h^{\min\{r+1,s\}} |\ln h|^{\overline{r}} \eps^{-\gamma} 
+h^{-\frac{d}{2}}(k+h^{\min\{r+1,s\}})\eps^{-(\sigma_1+2)} \Bigr). 
\end{align*}
Then there exists $\widetilde{\epsilon}_1>0$ such that for
$\epsilon\in(0,\widetilde{\epsilon}_1)$ there holds
\begin{equation}\label{eq4.8}
|u^{\epsilon}(x,t)|<1-\epsilon.
\end{equation}
Therefore, the assertion (i) follows from setting 
$\epsilon_1=\min\{\widehat{\epsilon}_1,\widetilde{\epsilon}_1\}$.

\medskip
{\em Step 2}:  By Theorem 6.1 of \cite{Bellettini_Paolini95}, we have
\begin{alignat}{2}\label{eq4.9}
u^{\epsilon}(x,t) &\geq 1-\epsilon
&&\qquad\forall x\in\{x\in\overline{\Omega};\, d(x,t)\geq \widehat{C}_0\epsilon^2|
\ln\,\epsilon|^2\},\\
u_{\epsilon}(x,t) &\leq -1+\epsilon
&&\qquad\forall x\in\{x\in\overline{\Omega};\, d(x,t)\leq -\widehat{C}_0\epsilon^2|
\ln\,\epsilon|^2\}.\label{eq4.10}
\end{alignat}
Repeating the argument of the proof of part (i) with \eqref{eq4.9} and 
\eqref{eq4.10} in place of \eqref{eq4.6} and \eqref{eq4.7}, respectively, 
we conclude that there exists $\epsilon_2>0$ such that the assertion (ii) holds.
The proof is complete.
\end{proof}

\section{Numerical experiments}\label{sec-5}
In this section, we present a couple two-dimensional numerical tests to gauge the performance 
of the proposed fully discrete IP-DG method with $r=1$. Both tests are done 
on the square domain $\Omega=[-1,1]^2$. 
In each numerical test, we first verify the spatial rate of convergence 
given in \eqref{eq3.36b} and \eqref{eq3.36c}, and the decay of the energy 
$J^h_\eps(u_h^\ell)$ defined in \eqref{eq3.12b} using $\epsilon=0.1$.  As expected, 
the energy decreases monotonically in the whole process of the evolution.
We then compute the evolution of the zero-level set of the solution of the 
Allen-Cahn problem with $\epsilon=0.125, 0.025, 0.005, 0.001$ and at various
time instances.

\medskip
$\mathbf{Test\, 1.}$ Consider the Allen-Cahn problem (\ref{eq1.1})-(\ref{eq1.4}) 
with the following initial condition:
{\small
\begin{equation*}
 u_0(x,y)=\begin{cases}
 \tanh(\frac{1}{\sqrt{2}\eps}(-\sqrt{(x-0.14)^2+(y-0.15)^2})), & \text{if}\ x> 0.14,0\leq y<-\frac{5}{12}(x-0.5),\\
 \tanh(\frac{1}{\sqrt{2}\eps}(-\sqrt{(x-0.14)^2+(y+0.15)^2})), & \text{if}\ x> 0.14,\frac{5}{12}(x-0.5)< y<0,\\
 \tanh(\frac{1}{\sqrt{2}\eps}(-\sqrt{(x+0.3)^2+(y-0.15)^2})), & \text{if}\ x<- 0.3,0\leq y<\frac{3}{4}(x+0.5),\\
 \tanh(\frac{1}{\sqrt{2}\eps}(-\sqrt{(x+0.3)^2+(y+0.15)^2})), & \text{if}\ x<- 0.3,-\frac{3}{4}(x+0.5)< y<0,\\
 \tanh(\frac{1}{\sqrt{2}\eps}(|y|-0.15)), & \text{if}\ -0.3\leq x\leq0.14,\\
 \tanh(\frac{1}{\sqrt{2}\eps}(\sqrt{(x-0.5)^2+y^2}-0.39)), & \text{if}\ x>0.14,y\geq-\frac{5}{12}(x-0.5)\\
 &\ \text{or}\ y\leq\frac{5}{12}(x-0.5),\\
  \tanh(\frac{1}{\sqrt{2}\eps}(\sqrt{(x+0.5)^2+y^2}-0.25)), &  \text{if}\ x<-0.3,y\geq-\frac{3}{4}(x+0.5)\\
 &\ \text{or}\ y\leq-\frac{3}{4}(x+0.5).\\
 \end{cases}
 \end{equation*}
}
Here $\tanh(x)=(e^x-e^{-x})/(e^x+e^{-x}).$ We note that $u_0$ can be written as 
\[
u_0=\text{tanh}\Bigl(\frac{d_0(x)}{\sqrt{2}\eps} \Bigr).
\]
Hence, $u_0$ has the desired form as stated in Proposition \ref{prop2.3}.

Table \ref{tab1} shows the spatial $L^2$ and $H^1$-norm errors and convergence rates, 
which are consistent with what are proved for the linear element in the convergence theorem. 
\begin{table}[htb]
\begin{center}
\begin{tabular}{|l|c|c|c|c|}
\hline
& $L^\infty(L^2)$ error & $L^\infty(L^2)$ order& $L^2(H^1)$ error& $L^2(H^1)$ order\\ \hline
$h=0.4\sqrt{2}$ & 0.25861 & & 1.02974 & \\ \hline
$h=0.2\sqrt{2}$ & 0.05338 &2.2764 & 0.49734 &1.0500 \\ \hline
$h=0.1\sqrt{2}$ & 0.01605 &1.7337& 0.25296 &0.9753 \\ \hline
$h=0.05\sqrt{2}$ & 0.00420 &1.9341 & 0.12693 &0.9949 \\ \hline
$h=0.025\sqrt{2}$ & 0.00104 &2.0138 & 0.06347 &0.9999 \\ \hline
\end{tabular}
\smallskip
\caption{Spatial errors and convergence rates of Test 1.} 
\label{tab1} 
\end{center}
\end{table}

Figure \ref{test1_energy} plots the change of the discrete energy
$J^h_\eps(u_h^\ell)$ in time, which should decrease according to \eqref{eq3.12b}.
This graph clearly confirms this decay property. 
\begin{figure}[tbh]
   \centering
   \includegraphics[width=3.5in,,height=2.0in]{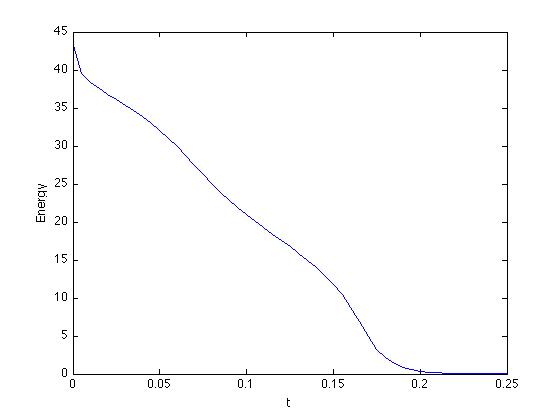} 
   \caption{Decay of the numerical energy $J^h_\eps(u_h^\ell)$ of Test 1.} \label{test1_energy}
\end{figure}

Figure \ref{figure1234} displays four snapshots at four fixed time points of the zero-level set of 
the numerical solution $u^{\epsilon,h,k}$ with four different $\epsilon$. They clearly indicate that
at each time point the zero-level set converges to the mean curvature flow $\Gamma_t$
as $\epsilon$ tends to zero. It also shows that the zero-level set evolves faster in time for 
larger $\epsilon$. 

\begin{figure}[th]
\centering
\includegraphics[height=1.8in,width=2.4in]{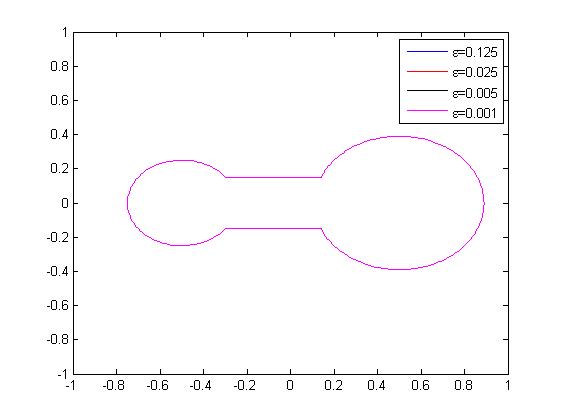}
\includegraphics[height=1.8in,width=2.4in]{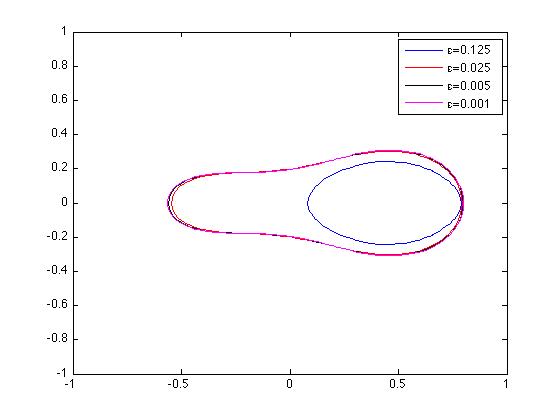}

\includegraphics[height=1.8in,width=2.4in]{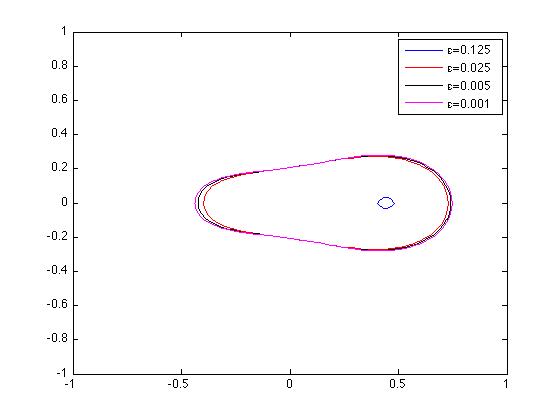}
\includegraphics[height=1.8in,width=2.4in]{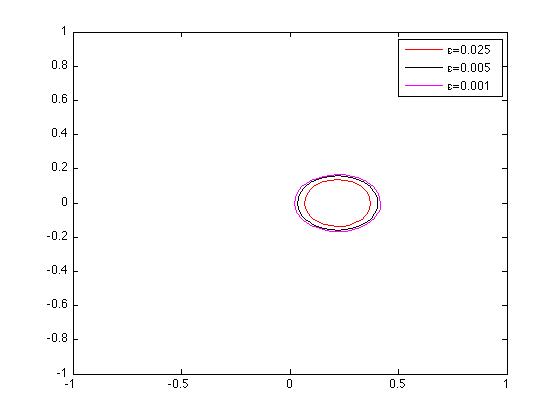}
\caption{Test 1: Snapshots of the zero-level set of $u^{\epsilon,h,k}$ at time $t=0, 
0.06, 0.09, 0.2$ and $\epsilon=0.125, 0.025, 0.005, 0.001$.}\label{figure1234}
\end{figure}

\medskip
$\mathbf{Test\, 2.}$ Consider the Allen-Cahn problem (\ref{eq1.1})-(\ref{eq1.4}) 
with the following initial condition:
\begin{equation*}
 u_0(x,y)=\begin{cases}
 \tanh(\frac{1}{\sqrt{2}\eps}(\text{min}\{d_1,d_2\})), & \text{if}\ 
\frac{x^2}{0.04}+\frac{y^2}{0.36}\geq1,\frac{x^2}{0.36}+\frac{y^2}{0.04}\geq1,\\
 &\ \text{or}\ \frac{x^2}{0.04}+\frac{y^2}{0.36}\leq1,\frac{x^2}{0.36}+\frac{y^2}{0.04}\leq1,\\
  \tanh(\frac{1}{\sqrt{2}\eps}(-\text{min}\{d_1,d_2\})), & \text{if}\ 
\frac{x^2}{0.04}+\frac{y^2}{0.36}<1,\frac{x^2}{0.36}+\frac{y^2}{0.04}>1,\\
 &\ \text{or}\ \frac{x^2}{0.04}+\frac{y^2}{0.36}>1,\frac{x^2}{0.36}+\frac{y^2}{0.04}<1.\\
\end{cases}
\end{equation*}
Here $d_1(x)$ and $d_2(x)$ stand for, respectively, the distance functions to the two ellipses.
We note that the above $u_0$ has the desired form as stated in Proposition \ref{prop2.3}.
 
Table \ref{tab2} shows the spatial $L^2$ and $H^1$-norm errors and convergence rates, 
which are consistent with what are proved for the linear element in the convergence theorem. 
\begin{table}[t]
\begin{center}
\begin{tabular}{|l|c|c|c|c|}
\hline
& $L^\infty(L^2)$ error & $L^\infty(L^2)$ order& $L^2(H^1)$ error& $L^2(H^1)$ order\\ \hline
$h=0.4\sqrt{2}$ & 0.09186 & & 0.29686 & \\ \hline
$h=0.2\sqrt{2}$ & 0.03670 &1.3237 &  0.16331&0.8622 \\ \hline
$h=0.1\sqrt{2}$ & 0.00911 &2.0103&  0.07603&1.1030 \\ \hline
$h=0.05\sqrt{2}$ & 0.00276 & 1.7228&  0.03740& 1.0235\\ \hline
$h=0.025\sqrt{2}$ & 0.00071 & 1.9588&  0.01846& 1.0186\\ \hline
\end{tabular}
\smallskip
\caption{Spatial errors and convergence rates of Test 2.} 
\label{tab2} 
\end{center}
\end{table}

Figure \ref{test2_energy} plots the change of the discrete energy
$J^h_\eps(u_h^\ell)$ in time, which should decrease according to \eqref{eq3.12b}.
This graph clearly confirms this decay property.
\begin{figure}[tbh]
   \centering
   \includegraphics[width=3.5in,,height=2.0in]{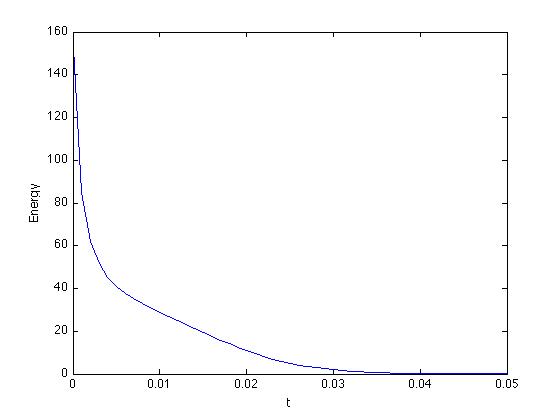} 
   \caption{Decay of the numerical energy $J^h_\eps(u_h^\ell)$ of Test 1.} \label{test2_energy}
\end{figure}

Like in Figure \ref{figure1234}, Figure \ref{figure5678} displays four snapshots at four fixed time 
points of the zero-level set of the numerical solution $u^{\epsilon,h,k}$ with four different $\epsilon$. 
Once again, we observe that at each time point the zero-level set converges to the mean curvature flow $\Gamma_t$
as $\epsilon$ tends to zero, and the zero-level set evolves faster in time for larger $\epsilon$.

\begin{figure}[th]
\centering
\includegraphics[height=1.8in,width=2.4in]{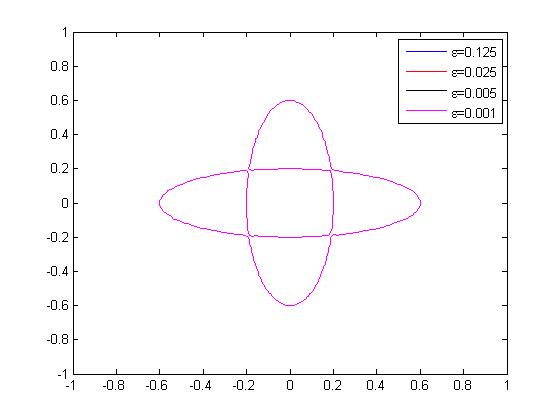}
\includegraphics[height=1.8in,width=2.4in]{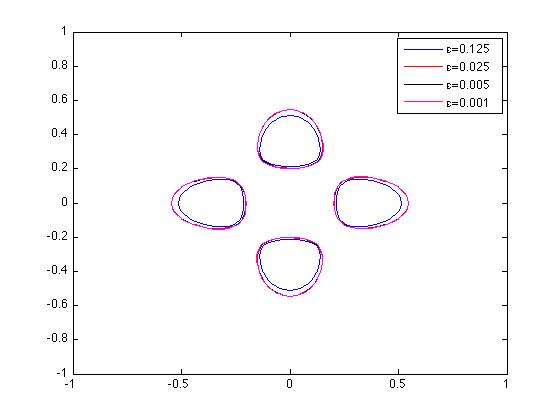}

\includegraphics[height=1.8in,width=2.4in]{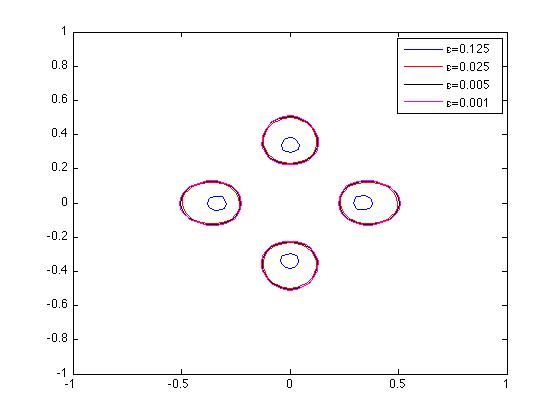}
\includegraphics[height=1.8in,width=2.4in]{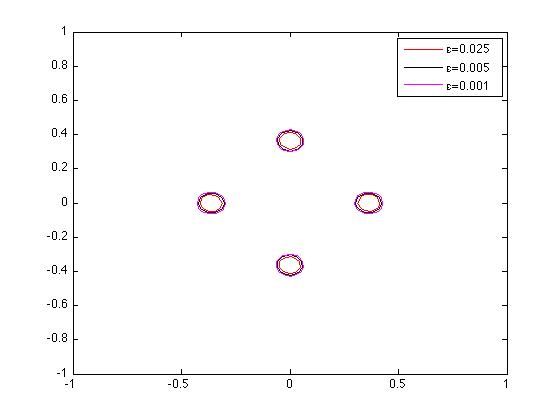}
\caption{Test 2: Snapshots of the zero-level set of $u^{\epsilon,h,k}$ at time $t=0, 
6\times 10^{-3}, 1.2\times 10^{-2}, 2\times 10^{-2}$ and $\epsilon=0.125, 0.025, 0.005, 0.001$.}\label{figure5678}
\end{figure}


\end{document}